\documentclass[twoside,a4paper,reqno,11pt]{amsart} 
\usepackage[top=28mm,right=30mm,bottom=28mm,left=30mm]{geometry}

\usepackage{amsfonts, amsmath, amssymb, mathrsfs, bm, latexsym, stmaryrd, array, hyperref, mathtools, bold-extra, xcolor}
\usepackage{verbatim}
\usepackage{float}
\restylefloat{table}

\newcommand{\leqs}{\leqslant}
\newcommand{\geqs}{\geqslant}

\newcommand{\vs}{\vspace{2mm}}

\newcommand{\Aut}{\operatorname{Aut}}

\newcommand{\Un}{{\rm U}}
\newcommand{\Li}{{\rm L}}

\makeatletter
\newcommand{\imod}[1]{\allowbreak\mkern4mu({\operator@font mod}\,\,#1)}
\makeatother

\newtheorem{theorem}{Theorem} 
\newtheorem*{conj*}{Conjecture}
\newtheorem*{thrm*}{Theorem}

\newtheorem{corol}[theorem]{Corollary}

\newtheorem{thm}{Theorem}[section] 
\newtheorem{prop}[thm]{Proposition} 
\newtheorem{lem}[thm]{Lemma}
\newtheorem{cor}[thm]{Corollary}

\theoremstyle{definition}
\newtheorem{remark}{Remark}
\newtheorem{rmk}[thm]{Remark}

\newtheorem*{example*}{Example}

\begin{document}

\author{Emily V. Hall}
\address{E.V. Hall, School of Mathematics, University of Bristol, Bristol BS8 1UG, UK}
\email{ky19128@bristol.ac.uk}

\title[]{The Quasiprimitive Almost Elusive Groups} 
\maketitle

\begin{abstract}
Let $G$ be a nontrivial transitive permutation group on a finite set $\Omega$ and recall that an element of $G$ is a derangement if it has no fixed points. Derangements always exist by a classical theorem of Jordan, but there are so-called elusive groups that do not contain any derangements of prime order. In a recent paper, Burness and the author introduced the family of almost elusive groups, which contain a unique conjugacy class of derangements of prime order. In this paper, we complete the classification of the quasiprimitive almost elusive groups.
\end{abstract}

\section{Introduction}
Let $G\leqs{\rm Sym}(\Omega)$ be a transitive permutation group with $|\Omega|\geqs 2$ and point stabiliser $H$. A classical theorem of Jordan \cite{Jordan} from 1872, which is an easy consequence of the orbit counting lemma, shows that $G$ always contains elements that act fixed point freely on $\Omega$. We call such elements \emph{derangements}. In recent years derangements have been widely studied and they emerge naturally in several different contexts. For example, Serre \cite{serre} discusses some interesting applications of Jordan's theorem in a wide range of different areas of mathematics. 

A major focus of research in this area concerns the existence of derangements with prescribed properties. For example, an extension of Jordan's theorem due to Fein, Kantor and Schacher \cite{FKS} from 1981 shows that $G$ always contains a derangement of prime power order. The proof given in \cite{FKS} requires the classification of finite simple groups, which is in clear contrast to the elementary proof of Jordan's theorem. Interestingly, although we are guaranteed the existence of derangements of prime power order, this does not always mean we can find derangements of prime order. For example, the sporadic group ${\rm M}_{11}$ with its primitive action on 12 points (that is, its action on the right cosets of a maximal subgroup ${\rm PSL}_2(11)$) does not contain a derangement of prime order (it does however contain derangements of order 4 and 8). Following \cite{CGJKKMN}, we say that a transitive group is \emph{elusive} if it contains no derangements of prime order. Although a complete classification of the elusive groups is currently out of reach, a large class of examples were classified by Giudici \cite{G}. His main theorem states that if $G$ is an elusive group with a transitive minimal normal subgroup, then $G={\rm M}_{11}\wr K$ with its product action on $\Delta^k$, where $|\Delta|=12$, $k$ is a positive integer and $K\leqs S_k$ is transitive. 

Notice that the set of derangements in $G$ is a normal subset, so it is a union of conjugacy classes. Therefore, it is natural to consider the number of conjugacy classes of derangements. It is shown by Burness and Tong-Viet in \cite{BTV1} that a primitive permutation group $G$ with point stabiliser $H$ has a unique conjugacy class of derangements if and only if it is either sharply 2-transitive, or $(G,H)=(A_5,D_{10})$ or $(\Li_2(8){:}3,D_{18}{:}3)$. This result was later extended by Guralnick \cite{Gur}, where he shows that a transitive group has a unique class of derangements only if it is primitive.

Motivated by the work of Burness, Tong-Viet and Guralnick, a natural extension of elusivity was introduced in \cite{BHall}. We say that a transitive permutation group is \emph{almost elusive} if there exists exactly one conjugacy class of derangements of prime order. For example if $n$ is a prime power, then the symmetric group $S_n$ with its natural action on $n$ points is almost elusive. If $G$ is a finite group and $H$ is a core-free subgroup, then it will be convenient to say that $(G,H)$ is almost elusive if the natural transitive action of $G$ on the cosets of $H$ is almost elusive. 

Recall that a permutation group is \emph{quasiprimitive} if every nontrivial normal subgroup is transitive. For instance, if $G$ is an almost simple group with socle $G_0$ and point stabiliser $H$ then $G$ is quasiprimitive if and only if $G=G_0H$. Note that every primitive group is quasiprimitive. In \cite[Theorem 1]{BHall}, using a version of the O'Nan-Scott theorem for quasiprimitive groups due to Praeger \cite{P93}, it is shown that every quasiprimitive almost elusive group is either almost simple or 2-transitive affine. Moreover, in \cite{BHall} the primitive almost simple groups with socle an alternating, a sporadic or a rank one group of Lie type were handled, with the remaining classical groups treated in \cite{Hall}. Therefore, to complete the classification of the primitive almost elusive groups, it remains to consider the exceptional groups of Lie type and the 2-transitive affine groups.

Our first main result is Theorem \ref{t:main1}. In the statement, we exclude the groups with socle $G_2(2)'$ and ${}^2G_2(3)'$ in view of the isomorphisms $G_2(2)'\cong \Un_3(3)$ and ${}^2G_2(3)'\cong\Li_2(8)$. For further information on the cases that arise in Theorem \ref{t:main1}, we refer the reader to Tables \ref{tab:fullastab2} in Section \ref{s:tables} and Remark \ref{r:rmkfullastab1}.

\begin{theorem}\label{t:main1}
Let $G\leqs {\rm Sym}(\Omega)$ be an almost simple primitive permutation group with point stabiliser $H$ and socle $G_0$, an exceptional group of Lie type. Then $G$ is almost elusive if and only if $(G,H)$ is one of the following $$(G_2(4).2,{\rm J}_2{:}2),\, ({}^2F_4(2)',\Li_2(25)),\,({}^2F_4(2),\Li_2(25).2_3) \mbox{ or }({}^2F_4(2),5^2{:}4S_4).$$
\end{theorem}

Next let $G=V{:}H\leqs {\rm AGL}(V)$ be a 2-transitive affine group with socle $V= (\mathbb{F}_p)^d$ and point stabiliser $H\leqs {\rm GL}_d(p)$, where $p$ is a prime and $d\geqs 1$. Here the 2-transitivity of $G$ implies that $H$ acts transitively on the non-zero vectors in $V$. These groups were classified by Hering \cite{Hering} and this is a key ingredient in our proof of Theorem \ref{t:mainaffine}. In part (iii) of Theorem \ref{t:mainaffine} we write $\mathcal{P}(n,i)$ for the $i^{\rm th}$ primitive group of degree $n$ in the Database of Primitive Groups in {\sc Magma} \cite{Mag}. For example, $\mathcal{P}(2^4,17)=2^4{:}{\rm Sp}_4(2)'$. Additionally, we write $\Gamma{\rm L}_m(q)$ for the general semilinear group of dimension $m$ over $\mathbb{F}_q$, where $q$ is a $p$-power.

\begin{theorem}\label{t:mainaffine}
Let $G=V{:}H$ be a 2-transitive affine group such that $|V|=n=p^d$ with $p$ prime and $d\geqs 1$. Then $G$ is almost elusive if and only if one of the following holds
\begin{itemize}
\item[{\rm (i)}] $H\leqs\Gamma{\rm L}_1(p^d)$.
\item[{\rm (ii)}] ${\rm SL}_2(q)\trianglelefteqslant H\leqs \Gamma{\rm L}_2(q)$, where $p=2$, $d$ is even and $q=2^{d/2}$.
\item[{\rm (iii)}] $G=\mathcal{P}(n,i)$, where $(n,i)$ is contained in Table \ref{tab:affine}.
\end{itemize}
\end{theorem}

\begin{table}[h]
\[
\begin{array}{lllll} \hline
n & i & &n & i \\  \hline\rule{0pt}{2.5ex} 
2^4 &17,\,19 & & 11^2 & 36,\, 38,\, 42,\, 43,\, 44 \\
3^4 & 44,\,68,\,69,\,70,\,90,\,99 & & 19^2 & 73,\,80 \\
3^6 & 145,198,239,240,366 &  & 23^2 &49,\,51\\
5^2& 12,\, 14,\, 17, \, 19 &  & 29^2 & 97,\, 103,\, 104 \\
7^2 &22,\,23,\,29 & & 59^2 &79,\,84 \\
\hline
\end{array}
\]
\caption{The almost elusive affine groups $G=\mathcal{P}(n,i)$}
\label{tab:affine}
\end{table}

By combining Theorems \ref{t:main1} and \ref{t:mainaffine} with the main results in \cite{BHall,Hall}, we now complete the classification of the almost elusive primitive groups (see Section \ref{s:tables} for Tables \ref{tab:fullastab1} and \ref{tab:fullastab2}).

\begin{theorem}\label{t:fullaeprim}
Let $G\leqs {\rm Sym}(\Omega)$ be a finite primitive permutation group with point stabiliser $H$. Then $G$ is almost elusive if and only if either 
\begin{itemize}
\item[{\rm (i)}] $G$ is almost simple and $(G,H)$ is contained in Tables \ref{tab:fullastab1} or \ref{tab:fullastab2}; or
\item[{\rm (ii)}] $G=V{:}H$ is a 2-transitive affine group as described in Theorem \ref{t:mainaffine}.
\end{itemize} 
\end{theorem}

Let us now extend the classification in Theorem \ref{t:fullaeprim} from primitive groups to quasiprimitive groups. To do this, we need to determine the pairs $(G,H)$, where $G$ is an almost simple group with socle $G_0$, $H$ is a core-free non-maximal subgroup of $G$ with $G=G_0H$ and $G$ is almost elusive with respect to the natural action of $G$ on the cosets of $H$. We direct the reader to Section \ref{s:tables} for the relevant tables.

\begin{theorem}\label{t:mainquasi}
Let $G$ be an almost simple group with socle $G_0$ and let $H$ be a core-free non-maximal subgroup of $G$ such that $G=G_0H$. Then $(G,H)$ is almost elusive only if one of the following holds:
\begin{itemize}
\item[{\rm (i)}] $G_0=\Un_n(q)$ and $H$ stabilises a 1-dimensional non-degenerate subspace of the natural module, where $q$ is even and $n\geqs 5$ is a prime divisor of $q + 1$.
\item[{\rm (ii)}] $G_0=\Li_2(p)$ with $p\geqs 5$ prime and $(G,H)$ is recorded in Table \ref{tab:imprimitiveAE}.
\item[{\rm (iii)}] $(G,H)$ is recorded in Table \ref{tab:QuasiInA}.
\end{itemize}
\end{theorem}

\begin{remark} \label{r:quasirmk}
Here we provide some remarks on Theorem \ref{t:mainquasi}.
\begin{itemize}
\item[{\rm (a)}] {Suppose that $(G,H)$ is as in case (i) of Theorem \ref{t:mainquasi} and write $H<M$, where $M$ is the stabiliser of a 1-dimensional non-degenerate subspace of the natural module. Then $(G,M)$ arises in case U1 of Table \ref{tab:fullastab1}, with the relevant conditions on $n$ and $q$ recorded in Remark \ref{r:rmkfullastab1}(b).
As discussed in \cite[Section 5.2.3]{Hall}, we anticipate that no genuine almost elusive examples arise in this case (that is, no examples which satisfy all the number-theoretic constraints), which would allow us to remove case (i) in Theorem \ref{t:mainquasi}. We refer the reader to Remark \ref{r:caseU1} for further information about quasiprimitive groups in this case. }
\item[{\rm (b)}] For part (ii), if $(G,H)$ is a case recorded in Table \ref{tab:imprimitiveAE}, then $(G,K)$ is almost elusive for any subgroup $K$ of $G$ isomorphic to $H$. See Remarks \ref{r:L1L2rmk} and \ref{r:L4L5rmk} for more details. 
\item[{\rm (c)}] Let $(G,H)$ be any of the cases recorded in Table \ref{tab:QuasiInA}. Then $G$ has a subgroup $K$ with $H\cong K$ such that $(G,K)$ is almost elusive. In the table we record the total number of $G$-classes of subgroups isomorphic to $H$ such that $G=G_0H$, together with the number of these $G$-classes that give almost elusive examples. See Remark \ref{r:main} for more information on Table \ref{tab:QuasiInA}. 
\end{itemize}
\end{remark}

Take $G\leqs{\rm Sym}(\Omega)$ to be a quasiprimitive permutation group with point stabiliser $H$. Assume $G$ is an almost simple group with socle $G_0$, an alternating or sporadic group, such that $G$ is not elusive (that is $(G,n)\neq({\rm M}_{11},12)$). Let $r$ denote the largest prime divisor of $|\Omega|$.  In \cite[Corollary 1.2]{BGW}, Burness, Giudici and Wilson prove that if $G$ is primitive then $G$ contains a derangement of order $r$. By inspecting the cases that arise in Theorems \ref{t:fullaeprim} and \ref{t:mainquasi}, we can establish the following extension. As in Theorem \ref{t:mainquasi} the case in which $G_0=\Un_n(q)$ and $H$ stabilises a 1-dimensional non-degenerate subspace of the natural module arises as a special case in Corollary \ref{c:largestprim}. See Section \ref{s:corollaries} for the proof.

\begin{corol}\label{c:largestprim}
Let $G\leqs{\rm Sym}(\Omega)$ be a quasiprimitive almost elusive permutation group with socle $G_0$ and point stabiliser $H$. Assume $G$ has derangements of prime order $s$. Then either $s$ is the largest prime divisor of $|\Omega|$ or one of the following holds
\begin{itemize}
\item[{\rm (i)}] $G$ is primitive, $(s,r)=(2,3)$ and $(G,H)=({}^2F_4(2)',\Li_2(25))$ or $({}^2F_4(2),\Li_2(25).2_3)$.
\item[{\rm (ii)}] $G$ is imprimitive and one of the following holds
\begin{itemize}
\item[{\rm (a)}] $G_0=\Un_n(q)$ and $H$ is properly contained in the stabiliser of a 1-dimensional non-degenerate subspace of the natural module, where $q$ is even and $n\geqs 5$ is a prime divisor of $q+1$.
\item[{\rm (b)}] $s=2$ and $(G,H)$ is as in case II or III of Table \ref{tab:imprimitiveAE}.
\item[{\rm (c)}] $s=3$ and $(G,H)=(\Li_2(p),C_p{:}C_d)$ is as in case IV of Table \ref{tab:imprimitiveAE} such that $\frac{(p-1)}{d}$ is divisible by an odd prime. 
\end{itemize}
\end{itemize}
\end{corol}

Let $G$ be an almost simple group with socle $G_0$ and let $H$ be a core-free subgroup of $G$ such that $G=G_0H$ and $(G,H)$ is almost elusive. We define the \emph{depth of $H$}, denoted $d_G(H)$, to be the longest possible chain of subgroups
\begin{equation}\label{e:gchain}
 G>L_1>\dots>L_{\ell-1}>L_\ell=H,
\end{equation}
such that $(G,L_i)$ is almost elusive for all $1\leqs i \leqs \ell$. Here we refer to $\ell$ as the length of the chain in \eqref{e:gchain}. We define the \emph{almost elusive depth} of $G$ to be 
\[
D_G=\max d_G(H)
\]
where we take the maximum over all core-free subgroups $H$ of $G$ such that $G=G_0H$ and $(G,H)$ is almost elusive. 
In the following corollary we let $\omega(n)$ and $\pi(n)$ denote the total number of prime divisors  and the number of distinct prime divisors of a positive integer $n$, respectively. 
The proof is presented in Section \ref{s:corollaries} and we refer the reader to Remark \ref{r:main} for more information the groups that arise.

\begin{corol}\label{c:depth}
Let $G$ be an almost simple group with socle $G_0$ and set $k=|G:G_0|$. If $D_G\geqs 2$ then one of the following holds:
\begin{itemize}
\item[{\rm (i)}] $G_0=\Un_n(q)$ where $q$ is even and $n\geqs 5$ is a prime divisor of $q+1$. 
\item[{\rm (ii)}] $D_G = \omega(k(p-1)/2)-\pi(k(p-1)/2) +1$ and $G_0=\Li_2(p)$, where $p=2^m-1$ is a prime.
\item[{\rm (iii)}] $D_G = \omega((p-1)/2)-\pi((p-1)/2) +1$ and $G=\Li_2(p)$, where $p=2.3^a-1$ is a prime with $a\geqs 2$.
\item[{\rm (iv)}] $D_G = \omega(p+1)-\pi(p+1)+1$ and $G={\rm PGL}_2(p)$ where $p=2^m+1$ is a prime.
\item[{\rm (v)}] $D_G=2$ and $G={\rm M}_{10},A_9,S_9,\Li_2(8).3, \Un_5(2).2$ or ${\rm PSp}_6(2)$.
\item[{\rm (vi)}] $D_G=3$ and $G= \Li_2(49).2_3$.
\item[{\rm (vii)}] $D_G=4$ and $G={\rm U}_4(2)$, ${\rm U}_4(2).2$ or ${\rm U}_3(3).2$.
\end{itemize}
\end{corol}

\begin{remark}Here we provide some remarks on Corollary \ref{c:depth}.
\begin{itemize}
\item[{\rm (a)}] {For a group to arise in case (i) with $D_G\geqs 1$ we need $G$ to be as in case U1 of Table \ref{tab:fullastab1}. That is $G=G_0.[2f]$ where $q=2^f$ and all the relevant number theoretic conditions are satisfied. As highlighted above, we do not anticipate any groups to arise in this case (see for example Remark \ref{r:rmkfullastab1}(b)). However, let us observe that if $G=G_0.C_{2f}$ is a genuine example in case U1, then $D_G\geqs \omega(2f)-1$.}
\item[{\rm (b)}] We do not know if $D_G$ can be arbitrarily large, which seems to depend on some very difficult open problems in number theory. But we have checked computationally that if $G_0=\Li_2(p)$ and $p<2^{1020}$ is a prime of the required form, then $D_G\leqs 10$. 
\end{itemize}
\end{remark}

To conclude the introduction, let us briefly describe the layout of the paper. In Section \ref{s:except} we complete the classification of the almost elusive almost simple primitive groups by proving Theorem \ref{t:main1}.
Here $G$ is an almost simple exceptional group of Lie type with socle $G_0$ and $H$ is a core-free maximal subgroup of $G$. We begin by presenting several preliminary results of a number-theoretic flavour and we briefly discuss the subgroup structure of exceptional groups. We also present Theorem \ref{t:pi}, which classifies the pairs $(G_0,H)$ such that $\pi(G_0)-\pi(H_0)\leqs 1$, where $H_0=H\cap G_0$ and $\pi(X)$ denotes the number of distinct prime divisors of $|X|$. This is an extension of \cite[Corollary 5]{LPS}, and it is the analogue of \cite[Theorem 2]{Hall}, which was stated for classical groups. Since $G$ is almost elusive only if $\pi(G_0)-\pi(H_0)\leqs 1$, this key result significantly reduces the number of cases we need to consider in the proof of Theorem \ref{t:main1}.

Next, in Section \ref{s:affine}, we prove Theorem \ref{t:mainaffine} on the 2-transitive affine groups, which completes the classification of the primitive almost elusive groups.
Here we apply Hering's theorem \cite{Hering}, which divides the 2-transitive affine groups into several infinite families together with a small number of sporadic cases. We use computational methods in {\sc Magma} \cite{Mag} for the sporadic cases that arise and the infinite families are handled case by case.

In Section \ref{s:quasi} we extend the primitive classification to all quasiprimitive groups. We easily reduce to the case where $G$ is almost simple with socle $G_0$ and the point stabiliser $H$ is contained in a core-free maximal subgroup $M$ such that $(G,M)$ is almost elusive. This allows us to use Theorem \ref{t:fullaeprim} to inspect the possibilities that arise. For the pairs of groups $(G,M)$ in Table \ref{tab:fullastab2} we can use computational methods in {\sc Magma}. For the infinite families in Table \ref{tab:fullastab1} we can apply similar techniques used for the primitive groups. For example, we can reduce the problem to the cases with $\pi(M\cap G_0)=\pi(H\cap G_0)$ (see Lemma \ref{l:piMpiH}).
Finally in Section \ref{s:corollaries} we present proofs of Corollaries \ref{c:largestprim} and \ref{c:depth} and then in Section \ref{s:tables} we present the relevant tables referred to in the statements of Theorems \ref{t:fullaeprim} and \ref{t:mainquasi}.

\vs

\noindent \textbf{Notation.} Our group theoretic notation is fairly standard. Let $A$ and $B$ be groups and $n$ a positive integer. We write $C_n$ or $n$ to denote a cyclic group of order $n$ and $[n]$ to denote an unspecified soluble group of order $n$. Additionally we use $D_n$ to denote the dihedral group of order $n$. An unspecified extension of $A$ by $B$ will be denoted as $A.B$, and we sometimes use $A{:}B$ if the extension splits. For simple groups we adopt the notation of Kleidman and Liebeck \cite{KL}. 
For instance, 
$$ {\rm PSL}_n(q)=\Li_n(q)=\Li_n^{+}(q),\,\,\,{\rm PSU}_n(q)=\Un_n(q)=\Li_n^{-}(q)$$
Additionally, let $G$ be a finite group and let $H$ be a core-free subgroup.  For a given property $X$ of a permutation group, we say $(G,H)$ has property $X$ if $G$ has property $X$ with respect to the natural action of $G$ on the cosets of $H$. For example, we say $(G,H)$ is almost elusive if $G$ is almost elusive with respect to the natural action of $G$ on the cosets of $H$. We let $\pi(A)$ and $\pi(n)$ denote the number of distinct prime divisors of $|A|$ and $n$ respectively, while $\alpha(A)$ and $\alpha(n)$ denote the set of distinct prime divisors of $|A|$ and $n$ respectively. For positive integers $a$ and $b$ we use $(a,b)$ to denote the greatest common divisor of $a$ and $b$.
\vs 

\noindent \textbf{Acknowledgments.}  I would like to thank my Ph.D. supervisor Professor Tim Burness for his guidance and helpful
comments. I would also like to acknowledge the financial support of EPSRC and the Heilbronn Institute for
Mathematical Research.

\section{Exceptional Groups}\label{s:except}

In this section we prove Theorem \ref{t:main1}. Throughout this section we let $G$ denote an almost simple permutation group with socle $G_0$, an exceptional group of Lie type over $\mathbb{F}_q$, and let $H$ denote a point stabiliser in $G$. Additionally, throughout the section we write $q=p^f$, with $p$ prime and $f\geqs 1$. We note that the groups with 
\begin{equation}\label{e:smallcases}
G_0\in\{{}^2F_4(2)',G_2(2)',{}^2G_2(3)'\}
\end{equation}
 are very amenable to simple computational methods for the proof of Theorem \ref{t:main1}. In particular, for Theorem \ref{t:main1} the groups $G$ with a socle in \eqref{e:smallcases} have already been handled in \cite[Propositions 4.4, 4.12 and 5.1]{BHall} and the almost elusive cases here can be found in Table \ref{tab:fullastab2}.
We note the groups with socle $G_2(2)'\cong\Un_3(3)$ or ${}^2G_2(3)'\cong\Li_2(8)$ appear in Table \ref{tab:fullastab2} as $\Un_3(3)$ or $ \Li_2(8)$ respectively to avoid repetition due to isomorphisms.  Thus we assume for the remainder of the section that $G_0$ is not one of the groups in \eqref{e:smallcases}. 

\subsection{Preliminary results}\label{s:prelimex}
In this section we provide some preliminary results that will be useful in the proof of Theorem \ref{t:main1}. We will discuss the subgroup structure of the exceptional groups as well as the conjugacy classes of certain prime order elements. However, we will begin by discussing some useful number-theoretic results from the literature on primitive prime divisors. Let $n$ be a positive integer. A prime divisor of $q^n-1$ is said to be a \emph{primitive prime divisor} if it does not divide $q^i-1$ for all $1\leqs i<n$. We define 
$$P_q^n=\{r\mid r \mbox{ is a primitive prime divisor of } q^n-1\}$$ 
and we note that $P^{nf}_p\subseteq P^n_q$. The following result is a famous theorem of Zsigmondy \cite{Zsig} from the 1890s regarding the existence of primitive prime divisors.

\begin{thm}\label{t:zsig}
The set $P_q^n$ is non-empty unless either $(n,q)=(1,2),(6,2)$, or $n=2$ and $q=p$ is a Mersenne prime.
\end{thm}

\begin{lem}
\label{l:Lemma A.1}
Assume $r\in P_q^n$ is an odd prime and let $m$ be a positive integer. Then $r$ divides $q^m-1$ if and only if $n$ divides $m$. Additionally, $r=nd+1$ for some $d\geqs1$.
\end{lem}
\begin{proof}
A proof of the first part can be seen in \cite[Lemma A.1]{BG_book} for example, and the second is an easy consequence of Fermat's Little Theorem. 
\end{proof}

In this paper we will often be interested in the situation where $|P_q^n|=1$ for certain values of $n$ and $q=p^f$. However, in general, it is extremely difficult to give exact conditions on $n$, $p$ and $f$ such that $|P_q^n|=1$. This is in part due to the complexity of the Diophantine equations that arise when trying to solve this problem. For example, suppose $P_q^n=\{r\}$ and $n$ is an odd prime that does not divide $q-1$. Then $(q,n,r)$ must be a solution to 
$$\frac{q^n-1}{q-1}=r^l$$
and currently all possible integer solutions to this Diophantine equation are not known. However, the following result provides some restrictions on the possible solutions in the case when $n$ is divisible by 3, and in most cases this will be sufficient for our purposes. 

\begin{prop}\label{p:blnag}
Let $x,y$ and $b$ be integers such that $|x|,|y|>1$ and $b\geqs 2$. Suppose $(x,y,b)$ is a solution to 
\[
\frac{x^3-1}{x-1}=x^{2}+x+1=y^b.
\]
Then $(x,y,b)=(18,7,3)$ or $(-19,7,3)$.
\end{prop}
\begin{proof}
This is a special case of \cite[Proposition 1]{BL}.
\end{proof}

Although we cannot give exact conditions on $n$, $p$ and $f$ such that $|P_q^n|=1$, \cite[Lemma 2.4]{Hall} provides necessary conditions on $f$.
For example, if $n\geqs7$ then $|P_q^n|=1$ only if $\alpha(f)\subseteq\alpha(n)$, where $\alpha(x)$ denotes the set of prime divisors of a positive integer $x$. 
Additionally, if we assume $P_q^n=\{dn+1\}$, then for certain values of $n$ we can provide conditions on the positive integer $d$, see \cite[Lemmas 2.5, 2.9, 2.11 and 2.12]{Hall} for example. In particular, \cite[Lemma 2.9]{Hall} considers the possibilities for $d$, in the case $n=2^a3$ for some $a\geqs 0$. Here we provide an extension of this result for $n=3$ and $6$. 

\begin{lem}\label{l:ppd2a3}
Suppose $n\in\{3,6\}$ and $P_q^n=\{r\}$. Then either $r\geqs 8nf+1$ or $r=dnf+1$ and one of the following holds:
\begin{itemize}
\item[{\rm (i)}] $d=1$ and $(n,q)=(3,4),(6,3),(6,4),(6,5),(6,8)$ or $(6,19)$.
\item[{\rm (ii)}] $d=2$ and $(n,q)=(3,2)$ or $(6,23)$.
\item[{\rm (iii)}] $d=4$ and $(n,q)=(3,3)$.
\item[{\rm (iv)}] $d=6$ and $(n,q)=(3,7), (6,9)$ or $(6,11)$.
\item[{\rm (v)}] $d=7$ and $(n,q)=(6,7)$.
\end{itemize}
\end{lem}

Before we begin our discussion on exceptional groups and their subgroup structure we first state the following elementary observation which is an application of \cite[Lemma 2.2]{BTV}. We provide the details for completeness. We note that $x\in G$ is a derangement if and only if $x^G\cap H=\emptyset$, where $x^G$ denotes the conjugacy class of $x$.

\begin{lem}\label{l:classes}
Let $G\leqs {\rm Sym}(\Omega)$ be an almost simple quasiprimitive permutation group with socle $G_0$, and point stabiliser $H$. Define $H_0=H\cap G_0$. Let $r$ be a prime divisor of $|\Omega|$ and let $a_r$ and $b_r$ denote the number of $G_0$-classes and $H_0$-classes of elements of order $r$ respectively. Assume $a_r>b_r$. Then there exists a derangement of order $r$ in $G$. 
\end{lem}
\begin{proof}
Since we are assuming $a_r>b_r$, there exists an element $y\in G_0$ of order $r$ such that $y^{G_0}\cap H_0 =\emptyset$. Assume for contradiction that there are no derangements of order $r$ in $G$. Then in particular, $y^G\cap H\neq\emptyset$, say $y^g\in H$ for some $g\in G$. We note $G=G_0H$, since $G$ is quasiprimitive, so we may write $g=uh$ where $u\in G_0$ and $h\in H$. Thus $(y^u)^h\in H$, which implies that $y^u\in H^{h^{-1}}=H$. However $y^u\in G_0$, contradicting the fact that $y^{G_0}\cap H_0 =\emptyset$. Thus $y$ is a derangement of order $r$ in $G$.
\end{proof}

The structure and classification of the maximal subgroups, up to conjugacy, of many of the exceptional groups of Lie type are well documented. 
For the groups with socle $G_2(q)$, the maximal subgroups were determined up to conjugacy by Kleidman \cite{Kleid-g2q} for $q$ odd and Cooperstein \cite{Coop} for $q$ even (see also \cite[Tables 8.30, 8.41 and 8.42]{BHR}). 
The maximal subgroups for groups with socle ${}^2F_4(q)$ and ${}^3D_4(q)$ were determined in \cite{Mal} and \cite{Kleid-3d4} respectively.
Finally a convenient source for the maximal subgroups of groups with socle ${}^2G_2(q)$ and ${}^2B_2(q)$ is \cite[Tables 8.43 and 8.16]{BHR}, which are reproduced from \cite{Kleid-g2q} and \cite{Suz} respectively. 
In addition, we note that the maximal subgroups of groups with socle $F_4(q)$, $E_6(q)$ and ${}^2E_6(q)$ have recently been fully determined up to conjugacy by Craven in \cite{Crav}. For the remaining exceptional groups, those in which $G_0\in\{E_7(q), E_8(q)\}$, we provide the following theorem. In this theorem and for the remainder of this section we write $G=(\bar{G}_{\sigma})'$, where $\bar{G}$ is a simple algebraic group of adjoint type over $\bar{\mathbb{F}}_p$ and $\sigma$ is an appropriate Steinberg endomorphism of $\bar{G}$. Additionally, we note this theorem also holds true, with minor adjustments, for the cases $G_0\in\{G_2(q), F_4(q), {}^2E_6(q), E_6(q)\}$. This is a version of \cite[Theorem 8]{LSE}. 

\begin{thm}\label{t:redsub}
Let $G$ be an almost simple group with socle $G_0=(\bar{G}_{\sigma})'\in\{E_7(q), E_8(q)\}$. Let $H$ be a core-free maximal subgroup of $G$ and set $H_0=H\cap G_0$. Then one of the following holds:
\begin{itemize}
\item[{\rm (I)}] $H$ is a maximal parabolic subgroup;
\item[{\rm (II)}] $H = N_G(\bar{H}_{\sigma})$ and $\bar{H}$ is a $\sigma$-stable non-parabolic maximal rank subgroup of $\bar{G}$: the possibilities for $H$ are determined in \cite[Tables 5.1 and 5.2]{LSS}.
\item[{\rm (III)}] $H = N_G(\bar{H}_{\sigma})$, where $\bar{H}$ is a maximal closed $\sigma$-stable positive dimensional subgroup of $\bar{G}$ (not parabolic nor maximal rank). 
\item[{\rm (IV)}] $H$ is of the same type as $G$ over a subfield of $\mathbb{F}_q$.
\item[{\rm (V)}] $H$ is an exotic local subgroup (determined in \cite{CLSS}).
\item[{\rm (VI)}] $G_0=E_8(q)$, $p\geqs 7$ and $H_0=(A_5\times A_6).2^2$.
\item[{\rm (VII)}] $H$ is almost simple and is not of type ${\rm (III)}$ or ${\rm (IV)}$.
\end{itemize}
\end{thm}

The conjugacy classes of maximal parabolic subgroups (type (I)) for $G_0=E_7(q)$ and $E_8(q)$ are in bijective correspondence with the nodes of the corresponding Dynkin diagrams. In this paper we are interested in the prime divisors of the orders of these subgroups. To obtain the prime divisors of a maximal parabolic subgroup $P$, it is convenient to use the Levi decomposition $P=QL$, where $Q$ is the unipotent radical of $P$ and $L$ is a Levi subgroup. From here we can easily read off the prime divisors of $|L|$ using the Dynkin diagram (note $p$ is the only prime divisor of $|Q|$). For example, if $G=E_7(q)$ and $P=P_5$, then the prime divisors of $|L|$ must divide $|{\rm SL}_5(q)||{\rm SL}_3(q)|$.

Following \cite[Theorem 8]{LSS} the subgroups of type (III) can be partitioned into three cases as shown below:
\begin{prop}
Let $G$ and $H$ be as in Theorem \ref{t:redsub}, with $H$ of type {\rm (III)}. Then one of the following holds:
\begin{itemize}
\item[{\rm (i)}] $G_0=E_7(q)$, $p\geqs 3$ and $H_0=(2^2\times{\rm P}\Omega^+_8(q).2^2). S_3$ or ${}^3D_4(q).3$,
\item[{\rm (ii)}] $G_0=E_8(q)$, $p\geqs 7$ and $H_0={\rm PGL}_2(q)\times S_5$,
\item[{\rm (iii)}] $(G_0,{\rm soc}(H_0))$ is one of the cases listed in \cite[Table 3]{LSE}.
\end{itemize}
\end{prop}

We present a similar proposition for the subgroups of type (VII). Note that we use ${\rm Lie}(p)$ to denote the set of finite simple groups of Lie type defined over fields of characteristic $p$. The possibilities for $S={\rm soc}(H)$ have been significantly refined in recent years. By combining the main results in recent work of Craven \cite{Crav1,Crav2}, we get the following:
\begin{prop}\label{p:vii}
Let $G$ and $H$ be as in Theorem \ref{t:redsub}, with $H$ of type ${\rm (VII)}$ and ${\rm soc}(H)=S$. Then one of the following holds:
\begin{itemize}
\item[{\rm (i)}] $S\not\in {\rm Lie}(p)$ and the possibilities for $S$ are described in \cite[Tables 10.1-10.4]{LSF}; or 
\item[{\rm (ii)}] $S\in {\rm Lie}(p)$ and one of the following holds:
\begin{itemize}
\item[{\rm (a)}] $G_0=E_8(q)$ and either $S=\Li_2(q_0)$ with $q_0\leqs (2,q-1).1312$ or $$S\in\{\Li^{\epsilon}_3(3),\Li^{\epsilon}_3(4), \Un_3(8),{\rm PSp}_4(2)^{'}, \Un_4(2), {}^2B_2(8)\};$$
\item[{\rm (b)}] $G_0=E_7(q)$ and $S=\Li_2(q_0)$ with $q_0\in\{7,8,25\}$. 
\end{itemize}
\end{itemize}
\end{prop}
 
The list of possibilities for $S$ in (i) of Proposition \ref{p:vii} has been refined further, see Craven \cite{Crav3} and Litterick \cite{Lit}. However for our work the tables in \cite{LSF} are sufficient. 

The following result regarding $G_2(q)$ will be useful for the analysis of both the exceptional and the affine groups. Let $G=G_2(q)$ with $q\geqs 3$ and let $V$ denote the minimal module for $G_2(q)$ (recall we write $q=p^f$ where $p$ is a prime and $f$ is a positive integer). By minimal module for $G_2(q)$ we mean $V$ is an irreducible module of dimension $7-\delta_{2,p}$, where $\delta_{2,p}$ is the Kronecker delta. 
We note that $M={\rm SL}^{\epsilon}_3(q){:}2$ is a maximal subgroup of $G$ and so the index 2 subgroup ${\rm SL}^{\epsilon}_3(q)$ of $M$ naturally embeds in $G$. Additionally, we note that for matricies $A_1,\dots,A_n$ we use ${\rm diag}(A_1,\dots,A_n)$ to denote a block diagonal matrix with blocks $A_1,\dots,A_n$ and $A_i^{-T}$ to denote the inverse transpose of $A_i$.

\begin{lem}\label{l:g2onmin}
Let $G=G_2(q)$ with $q\geqs 3$ and let $V$ be the minimal module of $G$. Take $x\in H={\rm SL}^{\epsilon}_3(q)$ such that $x$ acts as the matrix $A$ on the natural $H$-module. Then up to conjugacy, $x$ acts on $V$ as
\[
\begin{cases}
{\rm diag}(A,A^{-T}) & \mbox{if } q \mbox{ is even}\\
{\rm diag}(A,A^{-T},1) & \mbox{otherwise}
\end{cases}.
\]
\end{lem}
\begin{proof}
We work with the algebraic groups $\bar{G}=G_2(k)$ and $\bar{H}={\rm SL}_3(k)$, where $k=\bar{\mathbb{F}}_p$. Let $\bar{V}=V\otimes k$ denote the minimal module of $\bar{G}$ and $\bar{W}$ the natural $\bar{H}$-module. Then ${\rm dim}\bar{V}= 7 -\delta_{2,p}$ and 
\[
\bar{V}\downarrow \bar{H} =
\begin{cases}
\bar{W}\oplus \bar{W}^* & \mbox{if } q \mbox{ is even}\\
\bar{W}\oplus \bar{W}^* \oplus 0 & \mbox{otherwise}
\end{cases},
\] 
where $\bar{W}^*$ denotes the dual of $\bar{W}$ and $0$ denotes the trivial $\bar{H}$-module. The result now follows immediately since $x$ acts as the matrix $A$ on $\bar{W}$ and so $x$ acts as $A^{-T}$ on $\bar{W}^*$.
\end{proof}

We now present some useful results regarding the conjugacy classes of certain prime order elements in exceptional groups of Lie type.

\begin{prop}\label{p:classesG2}
Let $G=G_2(q)$ with $q\geqs 3$ and for $i\in\{3,6\}$ let $s_i$ denote the largest primitive prime divisor of $q^i-1$. Suppose that $s_i$ is the unique primitive prime divisor of $q^i-1$. Then $G$ contains at least $\frac{s_i-1}{6}$ distinct classes of elements of order $s_i$.
\end{prop}
\begin{proof}
The proof for $i=6$ and $i=3$ are similar, so we only provide details in the case $i=6$. 
Assume that $s_6$ is the unique primitive prime divisor of $q^6-1$. We note that we may view $G$ as a subgroup of ${\rm GL}(V)$ where $V$ is the minimal module for $G_2(q)$. Let $M={\rm SU}_3(q){:}2$, which is a maximal subgroup of $G$ (see \cite[Tables 8.30, 8.41 and 8.42]{BHR} for example). Define $L:={\rm SU}_3(q)<M$ with natural module $W$ and take $x\in M$ to be an element of order $s_6$. Then $x\in L$ since $s_6$ does not divide $|M:L|$. The conjugacy of semisimple prime order elements in $L$ is uniquely determined by the set of eigenvalues of the elements acting on the natural module of $L$ (over an appropriate extension field). By \cite[Proposition 3.3.2]{BG_book}, there are $(s_6-1)/3$ distinct $L$-classes of elements of order $s_6$, each represented by an eigenvalue set $[\Lambda_j]$, where $\Lambda_j=\{\lambda_j,\lambda_j^{q^2},\lambda_j^{q^4}\}$ and $\lambda_j\in\mathbb{F}_{q^6}$ is an $s_6^{{\rm th}}$ root of unity (note $\Lambda_j\neq\Lambda_j^{-1}$). We note that $M=L.\langle \psi \rangle$ where $\psi$ is an automorphism acting as the inverse transpose. Thus the classes represented by $[\Lambda_j]$ and $[\Lambda_j^{-1}]$ are fused in $M$. In particular, there are $(s_6-1)/6$ distinct $M$-classes of elements of order $s_6$. By applying Lemma \ref{l:g2onmin}, it follows that there are at least $(s_6-1)/6$ distinct ${\rm GL}(V)$-classes of such elements and the result follows.
\end{proof}

\begin{prop}\label{p:classes}
Let $G\in\{F_4(q),{}^3D_4(q)\}$ and assume that $q^4-q^2+1=r$ is prime. Then there are $(q^4-q^2)/\alpha$ distinct $G$-classes of elements of order $r$ in $G$, where $\alpha= 4$ if $G={}^3D_4(q)$ and $\alpha=12$ if $G=F_4(q)$.
\end{prop}
\begin{proof}
We inspect the relevant tables in \cite{DM,shin,sho}. Assume first $G={}^3D_4(q)$. By inspection of \cite[Table 2.1]{DM} we see that the only semisimple elements of order $r$ are those labeled $s_{14}$. Then turning to \cite[Table 4.4]{DM} we see that $G$ contains precisely $\frac{1}{4}(q^4-q^2)$ distinct $G$-classes of elements of order $r$. Next let us assume that $G=F_4(q)$ and $q$ is even. Then \cite[Table II]{shin} shows that the elements labeled $h_{76}$ are the only semisimple elements of order $r$ and there are $\frac{1}{12}(q^4-q^2)$ such $G$-classes. Similarly, for $G=F_4(q)$ and $q$ odd, inspection of \cite[Tables 8 and 9]{sho} shows the required elements are labeled by $h_{99}$ and the result follows.
\end{proof}

\subsection{A Reduction Theorem}
We remind the reader that $H_0=H\cap G_0$ and $\pi(X)$ denotes the number of distinct prime divisors of $|X|$. Additionally recall that $x\in G$ is a derangement if and only if $x^G\cap H=\emptyset$. From this definition it is clear to see that $G$ is almost elusive only if $\pi(G_0)-\pi(H_0)\leqs 1$. In work by Liebeck, Praeger and Saxl, \cite[Corollary 5]{LPS}, the subgroups $M$ of a simple group $G$ such that $\pi(G)=\pi(M)$ are described. Here we present an extension of this result, which is an analog of \cite[Theorem 2]{Hall} for classical groups. This result is a key tool in the proof of Theorem \ref{t:main1}. 

\begin{thm}\label{t:pi}
Let $G\leqs {\rm Sym}(\Omega)$ be an almost simple primitive permutation group with point stabiliser $H$ and socle $G_0$, an exceptional group of Lie type over $\mathbb{F}_q$. Then $\pi(G_0)-\pi(H_0)\leqs 1$ if and only if one of the following holds:
\begin{itemize}
\item[{\rm (i)}] $\pi(G_0)=\pi(H_0)$ and $(G_0,H_0)=(G_2(2)',\Li_2(7)),(G_2(3),\Li_2(13))$ or $({}^2F_4(2)', \Li_2(25))$.
\item[{\rm (ii)}] $\pi(G_0)=\pi(H_0) +1$ and $(G_0,H_0)$ is found in Table \ref{tab:maintab2}.
\end{itemize}
\end{thm}

\begin{table}
\[
\begin{array}{clllcc} \hline
\mbox{Case}& G_0 & H_0 & \mbox{Conditions} & i & r \\ \hline\rule{0pt}{2.5ex} 
{\rm R1} & {}^2G_2(q) & 2\times\Li_2(q) &  & 6 & 7\\
{\rm R1'} & {}^2G_2(3)' & 2^3{:}7 &  & &3\\
{\rm R2'} &  & D_{14}&  & &3\\
{\rm R3'} & & D_{18} &  & 6 &7\\
{\rm G1 }&G_2(q) & {\rm J}_1 & q=11 & 6 &37  \\
{\rm G2} & & {\rm J}_2  & q=4 & 6 &13\\
{\rm G3} & & \Li_2(13)  & q=4 & 2 & 5\\
{\rm G4} & & 2^3.\Li_3(2)  & q=3 & 3 &13\\
{\rm G5} & &{\rm SL}_3(q){:}2 & & 6 &r\\
{\rm G6} &           &{\rm SU}_3(q){:}2 & & 3 &r \\
{\rm G7} &           &{}^2G_2(q) & q=3^f, f \mbox{ is odd}  &3 &r  \\
{\rm G1' }&G_2(2)' & [3^{1+2}]{:}(3^2-1) & & 3 &7  \\
{\rm G2' }& & 4.{\rm PGU}_2(3) & &3&7  \\
{\rm G3' }& & 4^2{:}S_3 & & 3 &7  \\
{\rm D1} & {}^3D_4(q) & [q^9]{:}({\rm SL}_2(q^3)\circ (q-1)).(2,q-1)&  &12  &r \\
{\rm D2} & & G_2(q) &  & 12 &r \\
{\rm D3} &  & \Li_2(q^3)\times\Li_2(q) & q \mbox{ even }  &12 &r \\
{\rm D4} & & ({\rm SL}_2(q^3)\circ{\rm SL}_2(q)).2& q \mbox{ odd}  &12 &r \\
{\rm D5} & & [2^{11}]{:}(7\circ{\rm SL}_2(q))& q=2  &12 &13 \\
{\rm D6} & & (7\circ {\rm SL_3(2)}).7.2 &   q=2  &12 &13 \\
{\rm D7} &  & 7^2.{\rm SL}_2(3) &   q=2 &12 &13 \\
{\rm F1} & F_4(q) & (2,q-1).\Omega_9(q) &  &12&r \\
{\rm F1'} &{}^2F_4(2)' & \Li_3(3){:}2 &   &4 &5 \\
{\rm F2'} & & A_6.2^2 &   & 12 &13 \\
{\rm F3'} & & 5^2{:}4A_4 &   & 12 &13 \\
\hline
\end{array}
\]
\caption{Cases in which $\pi(G_0)-\pi(H_0)\leqs 1$}
\label{tab:maintab2}
\end{table}

\begin{rmk}
The column in Table \ref{tab:maintab2} labeled $i$ indicates the extra condition that there exists a unique primitive prime divisor $r_i$ of $q^i-1$. The column labeled $r$ indicates the unique prime that divides $|G_0|$ and not $|H_0|$. If there is an entry in column $i$ then $r$ is precisely the unique primitive prime divisor of $q^i-1$. In particular, we note that if $i=6$ and $q\neq 19$ then $r=q^2-q+1$ and if $i=12$ then $r=q^4-q^2+1$ (this can be shown using Proposition \ref{p:blnag}).
\end{rmk} 

\begin{proof}[Proof of Theorem \ref{t:pi}]
This can be shown by a direct comparison of the orders of $|G_0|$ (see \cite[p.208]{MT}) and $|H_0|$ (see Section \ref{s:prelimex} for the relevant references).  
The analysis is similar in most instances, so we only provide details for a handful cases:
\begin{itemize}
\item[{\rm (a)}] $G_0={}^2F_4(q)$ and $H_0=a^{\pm}{:}12$ where $a^{\pm}=(q^2\pm\sqrt{2q^3}+q\pm\sqrt{2q}+1)$,
\item[{\rm (b)}] $G_0=E_6(q)$ and $H_0=E_6(q_0).((3,q-1),k)$ where $q=q_0^k$ with $k$ prime,
\item[{\rm (c)}] $G_0={}^3D_4(q)$ and $H_0=G_2(q)$,
\item[{\rm (d)}] $G_0=E_7(q)$ and $H$ is a $P_7$ parabolic subgroup,
\item[{\rm (e)}] $G_0=E_8(q)$ and ${\rm soc}(H)=S=\Li_2(q_0)\in{\rm Lie}(p)$ with $q_0\leqs (2,q-1)1312$.
\end{itemize}

First consider case (a) and note that $|G_0|=q^{12}(q^{6}+1)(q^4-1)(q^3+1)(q-1)$. It is easy to show that both $a^+$ and $a^-$ divide $q^{12}+1$. Thus all prime divisors of $|H_0|$ divide $6(q^{12}+1)$. Take $s_4$ and $s_{12}$ to be the largest primitive prime divisors of $q^4-1$ and $q^{12}-1$ respectively (these both exist by Theorem \ref{t:zsig} and note that $s_{12}$ divides $q^6+1$). By Lemma \ref{l:Lemma A.1} we have $s_4=4d_4+1$ and $s_{12}=12d_{12}+1$ for some positive integers $d_i$, so $s_4,s_{12}\geqs 5$. Additionally, by Lemma \ref{l:Lemma A.1} both $s_4$ and $s_{12}$ are prime divisors of $q^{12}-1$. Thus neither $s_4$ nor $s_{12}$ divide $|H_0|$, implying that $\pi(G_0)-\pi(H_0)\geqs 2$.

Now let us assume we are in case (b). Here we have $|G_0|=\frac{1}{d} q_0^{36k}\prod_{i\in I}(q_0^{ik}-1)$ with $I=\{2,5,6,8,9,12\}$ where $d=(3,q-1)$ and every prime divisor of $|H_0|$ must divide $q_0.\prod_{j\in J}(q_0^{j}-1)$ where $J=\{5,9,12\}$. 
It is clear that primitive prime divisors of $q_0^{12k}-1$ and $q_0^{9k}-1$ divide $|G_0|$ and do not divide $|H_0|$ since $12k,9k\geqs 18$.

Next consider the case (c). Here 
$$|G_0|=q^{12}(q^6-1)^2(q^4-q^2+1)\,\,\mbox{ and }\,\, |H_0|=q^6(q^2-1)(q^6-1).$$
First observe that $r$ is a primitive prime divisor of $q^{12}-1$ if and only if $r$ divides $q^4-q^2+1$. Thus the only prime divisors of $|G_0|$ that do not divide $|H_0|$ are the prime divisors of $q^4-q^2+1$. Therefore $\pi(G_0)-\pi(H_0)=1$ if and only if $q^4-q^2+1=r^l$ for some odd prime $r$ and some positive integer $l$. Using the substitution $x=-q^2$ in Proposition \ref{p:blnag} we deduce there are no solutions for $l\geqs 2$. Thus we must assume $q^4-q^2+1=r$, which leads to case D2 in Table \ref{tab:maintab2}. We note that there are solutions to the equation $q^4-q^2+1=r$ with $r$ prime. For example, we can take $q\in\{2,3,4,9\}$.

Next let us assume we are in case (d). In this case $|G_0|=\frac{1}{d}q^{63}\prod_{i\in I}(q^{i}-1)$ with $I=\{2,6,8,10,12,14,18\}$ and where $d=(2,q-1)$ and $H_0=QL$ where $Q$ is a $p$-group and the prime divisors of $|L|$ divide $|E_6(q)|$. Thus all prime divisors of $|H_0|$ divide $q(q^5-1)(q^8-1)(q^9-1)(q^{12}-1)$. Therefore primitive prime divisors of $q^{14}-1$ and $q^{18}-1$ divide $|G_0|$ and not $|H_0|$. 

Finally assume we are in case (e) and let $q_0=p^t$. Here
$$|G_0|=p^{120f}\prod_{i\in I}(p^{if}-1)\,\,\mbox{ and }\,\,|S|=\frac{1}{d}p^t(p^{2t}-1),$$
where $I=\{2,8,12,14,18,20,24,30\}$ and $d=(2,p^t-1)$. We note that $|H_0|$ divides $|S|dt$. The largest powers of 2 and 3 less than $(2,q-1)1312$ are $2^{10}$ and $3^8$ respectively. Therefore we may assume $t\leqs 10$, so $2t< 24f, 30f$. Thus primitive prime divisors of $p^{24f}-1$ and of $p^{30f}-1$ divide $|G_0|$ and not $|H_0|$. 
\end{proof}

\subsection{Proof of Theorem \ref{t:main1}}\label{ss:ProofOf1}
We are now in a position to prove Theorem \ref{t:main1}. First we note that the cases with $G_0={}^2G_2(q)$, ${}^2G_2(3)'$, ${}^2B_2(q)$, $G_2(2)'$ and ${}^2F_4(2)'$ were handled previously in \cite[Propositions 4.4, 4.12, 5.1 and Sections 4.3 and 4.4]{BHall}. Thus for the remainder of this section we may assume that $G_0$ is not one of these groups. We begin by handling some small groups with the aid of {\sc Magma} \cite{Mag}.

\begin{prop}\label{p:smalldim}
Let $G\leqs {\rm Sym}(\Omega)$ be an almost simple primitive permutation group with point stabiliser $H$ and socle $G_0=G_2(3),G_2(4),G_2(5)$ or ${}^3D_4(2).$ Then $G$ is almost elusive if and only if $(G,H)=(G_2(4).2, J_2.2)$.
\end{prop}
\begin{proof}
In {\sc Magma} \cite{Mag} we use the command \verb|AutomorphismGroupSimpleGroup| to obtain the group $\Aut(G_0)$ as a permutation group (not necessarily on $\Omega$). Using \verb|LowIndexSubgroups| we can construct the groups $G$ such that $G_0\leqs G\leqs\Aut(G_0)$. For each group $G$ we call \verb|MaximalSubgroups| which returns a set of conjugacy classes of maximal subgroups in $G$. Then for each of these maximal subgroups $H$ we check if it is core-free using the command \verb|Core|. We can additionally identify $H$ by checking its order and structure using commands such as \verb|GroupName| and \verb|NormalSubgroups| for example. We can then use the inbuilt {\sc Magma} functions \verb|Classes| and \verb|IsConjugate| to construct a function to determine if a group $G$ acting on the cosets of a subgroup $H$ is almost elusive. For this we compute the conjugacy classes of elements of prime order in $G$ and $H$ using the \verb|Classes| command. Finally, we repeatedly use \verb|IsConjugate| to determine the fusion of $H$-classes in $G$, which allows us to output the classes of prime order derangements. If a unique such class is outputted then we conclude that $(G,H)$ is almost elusive. 
\end{proof}

In the remainder of the section we show that there are no further examples of almost elusive groups. We recall that $(G,H)$ is almost elusive only if $\pi(G_0)-\pi(H_0)\leqs 1$. Thus Theorem \ref{t:pi} reduces the proof of Theorem \ref{t:main1} to the cases G1-G7, D1-D7 and F1 listed in Table \ref{tab:maintab2}. In addition, we note that cases G2-G4 and D5-D7 were handled in Proposition \ref{p:smalldim}.

In the following proposition we prove Theorem \ref{t:main1} for cases G1, G5-G7 in Table \ref{tab:maintab2}. First we state our notation for unipotent elements in ${\rm GL}(V)={\rm GL}_n(q)$ where $q=p^f$. Take $M={\rm GL}_n(q)$ and let $x\in M$ be an element of order $p$. The Jordan form of $x$ on $V$ is denoted as 
$$x=[J_p^{a_p},\dots,J_1^{a_1}],$$
a block diagonal matrix, where $J_k$ denotes a standard unipotent Jordan block of size $k$ and $a_k$ is its multiplicity (note $n=\sum_{i=1}^pia_i$). In particular, if $x,y\in M$, then $x$ and $y$ are conjugate in $M$ if and only if they have the same Jordan form on $V$ (see \cite[Lemma 3.1.14]{BG_book} for example).  

\begin{prop}\label{p:G2}
Let $G\leqs {\rm Sym}(\Omega)$ be an almost simple primitive permutation group with point stabiliser $H$ and socle $G_0$ as in case G1, G5, G6 or G7 in Table \ref{tab:maintab2}. Then $G$ is not almost elusive.
\end{prop}
\begin{proof}
 Let $V$ denote the minimal $G_0$-module. Here $G_0=G_2(q)$ and by Proposition \ref{p:smalldim} we may assume $q\geqs 7$. We note that in all cases there exists a unique primitive prime divisor, $r_i$, of $q^i-1$ where $i=6$ for cases G1 and G5, and $i=3$ for cases G6 and G7. We note that $r_i$ is also a primitive prime divisor of $p^{fi}-1$, so by Lemma \ref{l:Lemma A.1} we may write $r_i=ifd_i+1$, where $d_i$ is a positive integer. Additionally, we note that every element in $G_0$ of order $r_i$ is a derangement.

By Proposition \ref{p:classesG2}, $G_0$ contains at least $ifd_i/6$ conjugacy classes of derangements of order $r_i$. Note $|{\rm Out}(G_0):G_0|=(1+\delta_{(3,p)})f$. Thus $G$ contains at least $\beta$ distinct classes of derangements of order $r_i$ where 
\[
\beta :=
\begin{cases}
id_i/12 & \mbox{if } p=3\\
id_i /6 & \mbox{otherwise}.
\end{cases}
\] 
By Lemma \ref{l:ppd2a3}, $\beta\geqs 2$ if and only if $(q,i)\neq (8,6),(19,6)$. Thus we conclude $G$ is not almost elusive for cases G1, G6, G7 and for case G5 with $q\neq 8, 19$. It remains to handle the final cases for G5, where $H_0={\rm SL}_3(q){:}2$ and $q=8$ or 19. 

Assume first that $q=19$. Suppose $x\in H_0$ is an element of order $19$. Then $x\in{\rm SL}_3(q)$ must have Jordan form $[J_2,J_1]$ or $[J_3]$ on the natural 3-dimensional module. Thus by Lemma \ref{l:g2onmin}, $x$ has Jordan form $[J_2^2,J_1^3]$ or $[J_3^2,J_1]$ on $V$ respectively. By inspecting \cite[Table 1]{Law}, the elements in $G_0$ of order 19 are in the class labeled $\tilde{A}_1$ and have Jordan form $[J_3,J_2^2]$ on $V$. Therefore there exist derangements of order 19.

Finally assume $q=8$, so 3 is the unique primitive prime divisor of $q^2-1$. By \cite[Proposition 3.2.1]{BG_book} there is a unique class of elements of order 3 in $H_0$. It is easy to check using {\sc Magma} that there are two distinct classes of elements of order 3 in $G_0$. Thus by Lemma \ref{l:classes}, $G_0$ contains derangements of order 3. 
\end{proof}

\begin{prop}
Let $G\leqs {\rm Sym}(\Omega)$ be an almost simple primitive permutation group with point stabiliser $H$ and socle $G_0$ in cases D1-D4 in Table \ref{tab:maintab2}. Then $G$ is not almost elusive. 
\end{prop}
\begin{proof}
By Proposition \ref{p:smalldim} we may assume $q\geqs 3$. Here we are assuming $q^4-q^2+1=r$ is prime and in all cases any element in $G_0$ of order $r$ is a derangement (since $r$ divides $|G_0|$ but not $|H_0|$). By Proposition \ref{p:classes} there are $(q^4-q^2)/4$ distinct $G_0$-classes of derangements of order $r$ in $G_0$. Since $|\Aut(G_0):G_0|=3f$, where $q=p^f$, there are at least $(q^4-q^2)/12f$ distinct $G$-classes of derangements of order $r$. It is easy to check that $(q^4-q^2)/12f\geqs 2$ for all $q\geqs 3$. 
\end{proof}

\begin{prop}
Let $G\leqs {\rm Sym}(\Omega)$ be an almost simple primitive permutation group with point stabiliser $H$ and socle $G_0$ in case F1 in Table \ref{tab:maintab2}. Then $G$ is not almost elusive. 
\end{prop}
\begin{proof}
Here $H_0=(2,q-1).\Omega_9(q)$ and we note we may assume that $G$ contains no graph automorphisms by \cite{LSS} (also see \cite[Tables 7.1 and 7.2]{Crav}).  
For $q=2$ we proceed as in the proof of Proposition \ref{p:smalldim}, so we may assume $q\geqs 3$. 
In this case, $q^4-q^2+1=r$ is a prime and any element in $G_0$ of order $r$ is a derangement. By Proposition \ref{p:classes} there are $(q^4-q^{2})/12$ distinct $G_0$-classes of elements of order $r$ in $G_0$. Since $|G:G_0|\leqs f$, there are at least $(q^4-q^{2})/12f\geqs 2$ distinct $G$-classes of derangements. The result follows.  
\end{proof}

This completes the proof of Theorem \ref{t:main1}. Moreover, in view of \cite{BHall} and \cite{Hall}, this completes the proof of the classification of the primitive almost simple almost elusive groups.  

\section{Affine Groups}\label{s:affine}
 
In this section we prove Theorem \ref{t:mainaffine}, which classifies the primitive almost elusive affine groups. Take $p$ to be a prime and let $V=(\mathbb{F}_p)^d$ be a $d$-dimensional vector space over $\mathbb{F}_p$. An \emph{affine transformation} of $V$ is a map $t_{h,v}:V\longrightarrow V$ with $h\in{\rm GL}(V)$ and $v\in V$ such that $t_{h,v}(u):=hu+v$. These affine transformations form the \emph{affine general linear group}, which is denoted ${\rm AGL}(V)$ or ${\rm AGL}_d(p)$, which we can view as a permutation group on $V$. The socle of ${\rm AGL}(V)$ may be identified with the additive group on $V$, which is isomorphic to $(C_p)^d$, and we say $G$ is an \emph{affine group} on $V$ if 
$$V\trianglelefteqslant  G=V{:}H\leqs {\rm AGL}(V),$$ 
where $H\leqs{\rm GL}(V)$ is the stabiliser of the zero vector in $V$. Additionally, an affine group $G$ is primitive if and only if $H$ is an irreducible subgroup of ${\rm GL}(V)$. For the remainder of this section, we will assume that $G=V{:}H\leqs{\rm AGL}(V)$ is a primitive affine group and we will use $(v,h)$ to denote the elements of $G$ where $v\in V$ and $h\in H$. 

\subsection{Preliminary results}\label{ss:afineprelims}

Throughout this section we let $V^* = V\setminus \{0\}$. We begin by discussing the prime order derangements in affine groups. Note that $|V|=p^d$, so every prime order derangement has order $p$. Any element of the form $(v,1)\in G$ such that $v\in V^*$ is a derangement of order $p$, since it acts as a translation by $v$ on $V$. Additionally, every element in $H$ has fixed points. In fact, we can precisely determine when an element of $G$ is a prime order derangement.

\begin{lem}\label{l:dercond}
An element $(v,h)\in G$ is a derangement of order $p$ if and only if all the following conditions are satisfied:
\begin{itemize}
\item[{\rm (i)}] $h^p=1$,
\item[{\rm (ii)}] $v\in {\ker}(h^{p-1}+\dots +h+1)$,
\item[{\rm (iii)}] $v\not\in {\rm im}(h-1)$.
\end{itemize}
\end{lem}
\begin{proof}
Let $g=(v,h)$ be a nontrivial element of $G$. Since $G$ acts on $V$ via affine transformations it is easy to see that $g$ is a derangement if and only if $u\neq hu+v$ for all $u\in V$. That is $v\not\in {\rm im}(h-1)$. 
Now, $g^p=(v+hv+h^2v+\dots +h^{p-1}v,h^p)$, so $g$ has order $p$ if and only if $h^p=1$ and $(h^{p-1}+\dots +h+1)(v)=0$, where $h^{p-1}+\dots +h+1 \in {\rm Hom}(V,V)$. That is, $v\in {\rm ker}(h^{p-1}+\dots +h+1)$. The result follows. 
\end{proof}

By \cite[Theorem 1]{BHall}, every primitive almost elusive affine group is 2-transitive. We provide a proof of this elementary observation for completeness.  

\begin{lem}\label{l:2transred}
$G$ is almost elusive only if $G$ is 2-transitive.
\end{lem}
\begin{proof}
Assume $G$ is almost elusive. Recall that any element in $G$ of the form $(v,1)$ such that $v\in V^*$ is a derangement of order $p$, that is every nontrivial element of $V$ is a derangement of order $p$. 
Since $G$ is almost elusive the elements of this form must all be conjugate in G, that is $V^*=v^G$ for all $v\in V^*$. Additionally since, $G=VH$ and $V$ is abelian we have $v^G=v^{VH}=v^H$ for all $v\in V^*$. 
This implies that $H$ acts transitively on $V^*$, that is $G$ is 2-transitive. 
\end{proof}

The 2-transitive affine groups have been determined by Hering \cite{Hering}. Other convenient sources for this result are \cite[Appendix 1]{Lieb} and \cite[Section 7.3]{Cam}. 

\begin{thm}[Hering's Theorem]\label{t:Hering}
 Let $G=V{:}H$ be a 2-transitive affine group of degree $n=p^d$. Then $(n,H)$ is one of the cases in Table \ref{tab:2trans}.
\end{thm}

\begin{table}
\[
\begin{array}{clll} \hline
& n & H & \mbox{Conditions}\\ \hline\rule{0pt}{2.5ex} 
{\rm I} & p^d &H\leqs\Gamma{\rm L}_1(p^d) & \\
{\rm II} & q^a & {\rm SL}_a(q)\trianglelefteqslant H\leqs \Gamma{\rm L}_a(q) & a\geqs 2  \\
{\rm III} & q^a & {\rm Sp}_a(q)\trianglelefteqslant H & a\geqs 4 \mbox{ even}\\
{\rm IV} & q^6 & {\rm G}_2(q)' \trianglelefteqslant H & q \mbox{ even} \\
{\rm V} & 5^2, 7^2,11^2,23^2 & {\rm SL}_2(3)\trianglelefteqslant H &\\
{\rm VI}& 3^4 & 2^{1+4}\trianglelefteqslant H &\\
{\rm VII} & 3^4,11^2,19^2,29^2,59^2 & {\rm SL}_2(5)\trianglelefteqslant H &\\
{\rm VIII} & 2^4 & A_6, A_7 & \\
{\rm IX} & 3^6 & {\rm SL}_2(13) & \\
\hline
\end{array}
\]
\caption{Groups arising in Hering's Theorem.}
\label{tab:2trans}
\end{table}

For the remainder of Section \ref{ss:afineprelims} we will establish some results regarding the Jordan form of elements of order $p$ in $H$. We note that will use the same notation for Jordan form as described in Section \ref{ss:ProofOf1}. These results will be particularly useful when $H$ belongs to one of the infinite families in Hering's Theorem (see cases I-IV in Table \ref{tab:2trans}). 
For these infinite families we show that there are no almost elusive examples for cases III and IV in Hering's Theorem (see Propositions \ref{p:CaseIII} and \ref{p:CaseIV}), but that there do exist almost elusive groups for cases I and II (see Propositions \ref{p:CaseI} and \ref{p:CaseII}). 
Note that every subgroup appearing in II of Table \ref{tab:2trans} is 2-transitive (see \cite[pg. 55]{Dix-Mort}), and detailed information on the exact structure of the 2-transitive groups in case I can be found in \cite[Section 15]{Foulser}. Additionally, for the groups in cases V-IX we use computational methods in {\sc Magma}, see Proposition \ref{p:smalldims}.

\begin{prop}
Let $h\in H$ be an element of order $p$ with Jordan form $[J_p^{a_p},\dots,J_1^{a_1}]$ on $V$. Then there exists an element $v\in V^*$ such that $(v,h)\in G$ is a derangement of order $p$ if and only if $a_i>0$ for some $1\leqs i\leqs p-1$.
\end{prop}
\begin{proof}
Fix an $\mathbb{F}_p$-basis $\{v_1,\dots,v_d\}$ of $V$ such that $h=[J_p^{a_p},\dots,J_1^{a_1}]$. By Lemma \ref{l:dercond}, there exists a $v\in V^*$ such that $(v,h)\in G$ is a derangement of order $p$ if and only if there exists a $v\in V^*$ such that $v\in{\ker}(h^{p-1}+\dots +h+1)$ and $v\not\in {\rm im}(h-1)$.

We note that we can write any $v\in V^*$ as $v=c_1v_1+\dots+c_dv_d$ for some $c_1,\dots,c_d\in \mathbb{F}_p$, and that $h^i=[(J_p^i)^{a_p},\dots,(J_1^i)^{a_1}]$, with 

$$J_k^i= \begin{pmatrix} 1 & \binom{i}{1} & \binom{i}{2} &\dots & \binom{i}{k-1}\\
0 & 1 & \binom{i}{1} &\dots &\binom{i}{k-2}\\
0 & 0 & 1 & \dots &  \binom{i}{k-3}\\
\vdots &\vdots&  \vdots &\ddots  & \vdots\\
0&0 &0 & \dots&  1\end{pmatrix},$$
where we take $\binom{a}{b}=0$ if $a<b$ or $b\leqs 0$. 
 
Since $\sum_{i=t}^{p-1}\binom{i}{t}=\binom{p}{t+1}$ is divisible by $p$ for all $1\leqs t \leqs p-2$, it is easy to show that $v\in {\rm ker}(h^{p-1}+\dots +h+1)$ if and only if either $a_p=0$, or $c_{kp}=0$ for all $1\leqs k\leqs a_p$.
Similarly, $v\in{\rm im}(h-1)$ if and only if for all $1\leqs i\leqs p$ either $a_i=0$, or $c_{ki+t}=0$ for all $1\leqs k \leqs a_i$, where
\[
t=
\begin{cases}
\sum_{j=i+1}^pa_jj & \mbox{if } i<p\\
0 & \mbox{if } i=p
\end{cases}.
\]
Thus it is clear to see that ${\rm ker}(h^{p-1}+\dots +h+1)\geqs {\rm im}(h-1)$ with equality if and only if $a_i=0$ for all $1\leqs i\leqs p-1$. That is there exists $v\in V^*$ such that $v\in{\ker}(h^{p-1}+\dots +h+1)$ and $v\not\in {\rm im}(h-1)$ if and only if $a_i>0$ for some $1\leqs i\leqs p-1$. Thus the result follows.
\end{proof}

\begin{cor}\label{c:ae}
Let $G=V{:}H$ be a 2-transitive affine group of degree $p^d$. Then $G$ is almost elusive if and only if one of the following holds: 
\begin{itemize}
\item[{\rm (i)}] $|H|$ is indivisible by $p$; or
\item[{\rm (ii)}] Both $|H|$ and $d$ are divisible by $p$, and every $h\in H$ of order $p$ has Jordan form $[J_p^{d/p}]$ on $V$.
\end{itemize}
\end{cor}

Next we briefly discuss the embedding of ${\rm GL}_{d/k}(q^k)$ in ${\rm GL}_d(q)$, where $k\geqs 1$ is a divisor of $d$.
Let $V_{\#}$ be a $d/k$-dimensional vector space over $\mathbb{F}_{q^k}$. Then we may view $V_{\#}$ as a $d$-dimensional vector space $V$ over $\mathbb{F}_q$. Additionally, any $\mathbb{F}_{q^k}$-linear transformation of $V_{\#}$ is also an $\mathbb{F}_{q}$-linear transformation of $V$, which yields an embedding of ${\rm GL}_{d/k}(q^k)$ in ${\rm GL}_d(q)$. 
We state the following lemma, which is a direct consequence of the normal basis theorem (see, for instance, \cite[Theorem 2.35]{LN}). 

\begin{lem}\label{l:nbt}
Let $d$ and $k$ be positive integers, such that $k$ divides $d$. Consider the vector spaces $V=(\mathbb{F}_q)^d$ and $V_\#=(\mathbb{F}_{q^k})^{d/k}$, where $q=p^f$ with $p$ prime and $f\geqs 1$. Fix an $\mathbb{F}_{q^k}$-basis $\{v_1,\dots,v_{d/k}\}$ of $V_{\#}$. Then there exists a scalar $\lambda\in \mathbb{F}_{q^k}\setminus \mathbb{F}_q$ such that
$$\{\lambda v_1, \lambda^q v_1,\dots\dots , \lambda^{q^{k-1}}v_1,\dots,\lambda v_{d/k}, \lambda^q v_{d/k},\dots , \lambda^{q^{k-1}}v_{d/k}\}$$
is an $\mathbb{F}_q$-basis for $V$.
\end{lem}

The following result is \cite[Lemma 5.3.2]{BG_book}.

\begin{lem}\label{l:jf1}
Let $d$ and $k$ be positive integers, such that $k$ divides $d$. Let $h\in{\rm GL}_{d/k}(p^k)$ be an element of order $p$ with Jordan form $[J_p^{a_p},\dots,J_1^{a_1}]$ on $V_{\#}=(\mathbb{F}_{p^k})^{d/k}$. Then $h$ has Jordan form $[J_p^{ka_p},\dots,J_1^{ka_1}]$ on $V=(\mathbb{F}_p)^d$.
\end{lem}
\begin{proof}
Fix an $\mathbb{F}_{p^k}$-basis $\{v_1,\dots,v_{d/k}\}$ of $V_{\#}$ such that $h=[J_p^{a_p},\dots,J_1^{a_1}]\in{\rm GL}_{d/k}(p^k)$. Then by Lemma \ref{l:nbt} there is an $\mathbb{F}_p$-basis for $V$, 
$$\beta=\{\lambda v_1, \lambda^p v_1,\dots\dots , \lambda^{p^{k-1}}v_1,\dots,\lambda v_{d/k}, \lambda^p v_{d/k},\dots , \lambda^{p^{k-1}}v_{d/k}\},$$
such that $\lambda\in \mathbb{F}_{p^k}\setminus \mathbb{F}_p$. Since we know how $h$ acts on the basis vectors in $\{v_1,\dots,v_{d/k}\}$ and $h$ is linear over $\mathbb{F}_{p^k}$, it is easy to see how $h$ acts on the basis vectors in $\beta$. Then with respect to an appropriate ordering of the basis $\beta$, we have $h=[J_p^{ka_p},\dots,J_1^{ka_1}]$ as required. 
\end{proof}

The final results of this preliminary section concern the general semilinear group and we first recall the structure of this group. Let $d$ and $k$ be positive integers and fix a prime $p$. 
Let $\{u_1,\dots,u_d\}$ be a basis for the natural module $U=(\mathbb{F}_{p^k})^d$ of ${\rm GL}_d(p^k)$. 
We can define a map, $\gamma: U \mapsto U$, where
\begin{equation*}\label{e:phibasis}
\gamma:\sum_i \lambda u_i \mapsto \sum_i \lambda^p u_i.
\end{equation*}
Then $\gamma$ induces a field automorphism, $\phi:{\rm GL}_d(p^k)\mapsto {\rm GL}_d(p^k)$, where $(a_{ij})^\phi=(a_{ij}^p)$ for all $(a_{ij})\in {\rm GL}_d(p^k)$. The general semilinear group is defined to be $\Gamma{\rm L}_d(p^{k})={\rm GL}_d(p^k){:} \langle\phi \rangle$ and we write the elements of $\Gamma{\rm L}_d(p^{k})$ as $(g,\phi^l)$ where $g\in{\rm GL}_d(p^k)$ and $\phi^l \in \langle \phi \rangle$. We refer to the map $\phi$ as a standard field automorphism of order $k$ in ${\rm GL}_d(p^k)$. In general, we say an element in $\Gamma{\rm L}_d(p^{k})\backslash {\rm GL}_d(p^k)$ is a field automorphism and note that they have the form $(g, \phi^l)$ such that $1\leqs l \leqs k-1$. Additionally, we recall that if $k$ divides $d$ then ${\rm GL}_{d/k}(p^k)$ embeds in ${\rm GL}_d(p)$. Therefore since the map $\gamma$ acts as an $\mathbb{F}_p$-linear map, $\Gamma{\rm L}_{d/k}(p^{k})$ embeds in ${\rm GL}_d(p)$. 

\begin{lem}\label{l:fieldjf}
Let $G=\Gamma{\rm L}_{d/k}(p^k)$ where $p$ is a prime and $k=pf$ for some $f\geqs 1$. Let $x=(1,\psi)\in G$ be an element of order $p$. Then $x$ has Jordan form $[J_p^{d/p}]$ on $V$. 
\end{lem}
\begin{proof}
We begin by noting that the standard field automorphism $\phi$ has order $k=pf$ and so $\psi=\phi^{if}$ for some $1\leqs i\leqs p-1$. 
Fix an $\mathbb{F}_{p^k}$-basis $\{v_1,\dots,v_{d/k}\}$ for $V_\#=(\mathbb{F}_{p^{k}})^{d/k}$, such that it is compatible with $\phi$ as in the discussion above. 
By Lemma \ref{l:nbt} there is an $\mathbb{F}_p$-basis for $V$,
$$\beta=\{\lambda v_1,\lambda^p v_1,\dots \lambda^{p^{k-1}}v_1,\dots,\lambda v_{d/k},\lambda^p v_{d/k},\dots \lambda^{p^{k-1}}v_{d/k}\}$$
such that $\lambda\in\mathbb{F}_{p^{k}}\setminus \mathbb{F}_p$. 
Note that $\psi$ acts on the vectors of $\beta$ as follows, $\psi(\lambda^{p^l}v_m)=\lambda^{p^{l+if}}v_m$ for all $0\leqs l\leqs k-1$ and $1\leqs m\leqs {d/k}$.
We can partition the set $\beta$ into $d/k$ sets of size $k$, namely the sets $\{\lambda v_1,\lambda^p v_1,\dots \lambda^{p^{k-1}}v_1\},\dots,\{\lambda v_{d/k},\lambda^p v_{d/k},\dots \lambda^{p^{k-1}}v_{d/k}\}$. 
On each of these $d/k$ sets $\psi$ acts as $f$ disjoint $p$-cycles. For example, on the first set an example of such a $p$-cycle is $(\lambda v_1,\lambda^{p^{if}}v_1,\dots,\lambda^{p^{(p-1)if}}v_1)$. Thus $\psi$ acts on $\beta$ as a product of $d/p$ disjoint $p$-cycles. It then follows that $(1,\psi)$ has Jordan form $[J_p^{d/p}]$ on $V$. 
\end{proof}

We are now in a position to state results regarding the Jordan form of field automorphisms of order $p$ in the general semilinear group $\Gamma{\rm L}_{d/k}(p^k)$. 
 First we state a well known number theoretic result. 
We remind the reader that for positive integers $a$ and $b$, we use the notation $(a,b)$ to denote the greatest common divisor of $a$ and $b$.

\begin{lem}\label{l:cancel}
Suppose $a,b,k,n,m\in\mathbb{Z}$ such that $k,m\neq 0$ and $(k,m)=d$. Then $ka\equiv kb \imod m$ if and only if $a\equiv b \imod {\frac{m}{d}}$. Additionally, the linear congruence $kx\equiv n\imod m$ has solutions if and only if $d$ divides $n$. Moreover if $d$ divides $n$ then there exist exactly $d$ solutions modulo $m$. 
\end{lem}

\begin{lem}\label{l:gamma1}
Let $d$ be a positive integer and let $p$ be a prime divisor of $d$. Let $x\in\Gamma{\rm L}_1(p^{d})$ be a field automorphism of order $p$.
Then $x$ has Jordan form $[J_p^{d/p}]$ on $V=(\mathbb{F}_p)^d$.
\end{lem}
\begin{proof}
Recall $H=\Gamma{\rm L}_1(p^d)={\rm GL}_1(p^d){:}\langle \phi\rangle$, where $\phi$ is the standard field automorphism of order $d$. We write $x=(a,\psi)$ where $a\in {\rm GL}_1(p^d)$ and $\psi= \phi^j$ for some integer $1\leqs j< d$. Additionally we write $d=pk$ for some $k\geqs 1$. Since $x$ has order $p$, this implies that $\psi$ has order $p$ and $a\psi(a)\dots\psi^{p-1}(a)=1$. In particular, $j=ik$ with $1\leqs i\leqs p-1$.

We claim that $x$ is $H$-conjugate to $(1,\psi)$ and so the result follows from Lemma \ref{l:fieldjf}. In order to prove this claim we first note that an element $(b,\phi^t)\in H$ is $H$-conjugate to $(1,\psi)$ only if $\phi^t=\psi$. Thus we proceed by showing that $|(1,\psi)^H|$ is equal to the number of order $p$ elements in $H$ of the form $(b,\psi)$.

 We recall that $(b,\psi)\in H$ has order $p$ if and only if 
$$b\psi(b)\dots\psi^{p-1}(b)=b\phi^{ik}(b)\dots\phi^{(p-1)ik}(b)=1.$$ 
Using Lemma \ref{l:cancel} we can show that for each $1\leqs z\leqs p-1$ there exists a unique $1\leqs y\leqs p-1$ such that $zik\equiv yk\imod d$. Thus since $\phi$ has order $d$ the equation above is exactly equivalent to  
$$b\phi^k(b)\dots\phi^{(p-1)k}(b)=b^{1+p^k+\dots+p^{(p-1)k}}=1.$$ Since ${\rm GL}_1(p^d)$ is a cyclic group of order $p^d-1$ there are 
$$(p^d-1,1+p^k+\dots+p^{(p-1)k})=1+p^k+\dots+p^{(p-1)k}$$ 
many elements $b\in{\rm GL}_1(p^d)$ such that $(b,\psi)$ has order $p$. 

An element $(b,\phi^t)\in C_H((1,\psi))$ if and only if $b=\psi(b)$ which is equivalent to $b^{p^{ik}-1}=1$. There are exactly $(p^d-1,p^{ik}-1)=p^k-1$ such elements $b\in {\rm GL}_1(p^d)$. Thus $|C_H((1,\psi))|=(p^k-1)d$, so 
$$|(1,\psi)^H|=(p^d-1)d/(p^k-1)d=1+p^k+\dots+p^{(p-1)k},$$
 and the claim follows. 
The result now follows by applying Lemma \ref{l:fieldjf}.
\end{proof}

\begin{lem}\label{l:gaml2}
Let $p,k$ and $d$ be integers such that $p$ is a prime, $d$ is divisible by $k$ with $d/k\geqs 2$ and $k$ is divisible by $p$. Let $x\in\Gamma{\rm L}_{d/k}(p^{k})$ be a field automorphism of order $p$. Then $x$ has Jordan form $[J_p^{d/p}]$ on $V=(\mathbb{F}_p)^d$.
\end{lem}
\begin{proof}
Let $\phi$ denote the standard field automorphism of order $p$ in $\Gamma{\rm L}_{d/k}(p^{k})$. We note that all field automorphisms of order $p$ are contained a coset ${\rm GL}_{d/k}(p^{k})\phi^i$ for some $1\leqs i\leqs p-1$. Thus $x\in{\rm GL}_{d/k}(p^{k})\sigma$, where $\sigma=\phi^i$ for some $1\leqs i\leqs p-1$.  By the theory of Shintani descent (see \cite[Section 3.4]{BGH} for more details for example) there is a bijective correspondence between the set of ${\rm GL}_{d/k}(p^{k}){:}\langle \sigma \rangle $-classes in the coset ${\rm GL}_{d/k}(p^{k})\sigma$ and the set of conjugacy classes in ${\rm GL}_{d/k}(p^{k/p})$. In particular, the classes of elements of order $p$ in ${\rm GL}_{d/k}(p^{k})\sigma$ correspond to classes of elements of order 1 in ${\rm GL}_{d/k}(p^{k/p})$ (see \cite[Lemma 3.20]{BGH} for a proof of this). We conclude there is a unique ${\rm GL}_{d/k}(p^{k}){:}\langle \sigma \rangle $-class of elements of order $p$ in $\Gamma{\rm L}_{d/k}(p^{k})$. Therefore we may assume that $x=(1,\sigma)=(1,\phi^i)$, so by Lemma \ref{l:fieldjf} the result follows.  
\end{proof}

\subsection{Proof of Theorem \ref{t:mainaffine}}\label{ss:ProofOfAffine}

We now turn to the proof of Theorem \ref{t:mainaffine}. We remind the reader of the notation we will use throughout this section: $G=V{:}H\leqs {\rm AGL}(V)$ with $V=(\mathbb{F}_p)^d$ and $H$ is an irreducible subgroup of ${\rm GL}(V)$. Additionally, we recall that we may also assume that $G$ is 2-transitive (see Lemma \ref{l:2transred}). Thus we approach the proof by inspecting the cases in Hering's Theorem (see Theorem \ref{t:Hering} and Table \ref{tab:2trans}). 

\begin{prop}\label{p:smalldims}
Theorem \ref{t:mainaffine} holds for $G$ as in case V-IX of Table \ref{tab:2trans}. 
\end{prop}
\begin{proof}
This is a simple calculation using the Database of Primitive Groups in {\sc Magma} \cite{Mag}, which records the primitive groups up to degree 4095. The command \verb|PrimitiveGroups| outputs the groups with our desired degree. We can then use the \verb|Classes| command to obtain all the conjugacy class in $G$ of elements of prime order. Using the \verb|Fix| command, which outputs the set of fixed points of an element of our group, for each $G$-class we can find the number of fixed points of the elements. If there is a unique $G$-class of elements of prime order with no fixed points then we conclude the group is almost elusive. 
\end{proof}

It now remains to handle the infinite families in Hering's theorem, namely cases I-IV in Table \ref{tab:2trans}. We recall that the non-zero vectors in $V$ form a single class of derangements of order $p$. Thus $G$ is almost elusive if there exist no derangements of the form $(v,h)\in G$, where both $v$ and $h$ are nontrivial.

\begin{prop}\label{p:CaseI}
Assume $G$ is a 2-transitive group as in case I of Table \ref{tab:2trans}. Then $G$ is almost elusive.
\end{prop}
\begin{proof}
Here $n=p^d$ and $H\leqs\Gamma{\rm L}_1(p^d)={\rm GL}_1(p^d){:}d\leqs{\rm GL}_d(p)$. If $p$ does not divide $d$, then $p$ does not divide $|H|$ and thus $G$ is almost elusive by Lemma \ref{l:dercond}. Now assume that $p$ divides $d$. Any element of order $p$ in $\Gamma{\rm L}_1(p^d)$ must be a field automorphism, so the result follows by Lemma \ref{l:gamma1} and Corollary \ref{c:ae}.
\end{proof}

\begin{prop}\label{p:CaseII}
Assume $G$ is as in case II of Table \ref{tab:2trans}. Then $G$ is almost elusive if and only if $p=a=2$.
\end{prop}
\begin{proof}
Here ${\rm SL}_a(q)\trianglelefteqslant H\leqs \Gamma{\rm L}_a(q)$ and $n=p^d=q^a$ with $a\geqs 2$. We recall that $\Gamma{\rm L}_a(q)=\Gamma{\rm L}_{d/k}(p^{k})<{\rm GL}_d(p)$, where $k=d/a$. Define $V_\# = (\mathbb{F}_{q})^a$, an $a$-dimensional vector space over $\mathbb{F}_q$ (the natural module of ${\rm GL}_a(q)$). 
Assume first that $a\geqs 3$ and take $h\in {\rm SL}_a(q)\leqs H$ to be an element of order $p$ with Jordan form $[J_2,J_1^{a-2}]$ on $V_\#$. Then by Lemma \ref{l:jf1}, $h$ has Jordan form $[J_2^k,J_1^{k(a-2)}]$ on $V$. Thus Corollary \ref{c:ae} implies $G$ is not almost elusive. 
Finally assume that $a=2$. Take $h\in{\rm GL}_2(q)$ to be an element of order $p$. Then $h$ has Jordan form $[J_2]$ on $V_{\#}$ and $[J_2^{d/2}]$ on $V$. Suppose first $p\geqs 3$. Then $G$ is not almost elusive by Corollary \ref{c:ae}. Finally suppose $p=2$. Then using Lemma \ref{l:gaml2} we see that every element of order $2$ in $\Gamma{\rm L}_2(q)$ has Jordan form $[J_2^{d/2}]$ on $V$. Thus the result follows by Corollary \ref{c:ae}.
\end{proof}

\begin{prop}\label{p:CaseIII}
Assume $G$ is as in case III of Table \ref{tab:2trans}. Then $G$ is not almost elusive.
\end{prop}
\begin{proof}
In this case ${\rm Sp}_{a}(q)\trianglelefteqslant H$ and $p^d=q^{a}$ with $a\geqs 4$ even. Define $V_\# = (\mathbb{F}_{q})^{a/2}$ and let $h\in {\rm Sp}_{a}(q)$ be an element of order $p$ with Jordan form $[J_2,J_1^{a-2}]$ on $V_{\#}$. Then $h$ has Jordan form $[J_2^k,J_1^{k(a-2)}]$ on $V$, where $d=ak$. Thus by Corollary \ref{c:ae}, $G$ is not almost elusive. 
\end{proof}

\begin{prop}\label{p:CaseIV}
Assume $G$ is as in case IV of Table \ref{tab:2trans}. Then $G$ is not almost elusive.
\end{prop}
\begin{proof}
Here $p=2$ and $G_2(q)^{'}\trianglelefteqslant H$ with $2^d=q^6$. We note that we can handle the case $d=6$ easily in {\sc Magma} using the same method outlined in the proof of Proposition \ref{p:smalldims}. Thus we may assume that $d\geqs 12$ and $G_2(q)\trianglelefteqslant H$.

Note that ${\rm SL}_3(q){:}2$ is a maximal subgroup of $G_2(q)$ (see \cite[Table 8.30]{BHR} for example). Let $W$ denote the natural ${\rm SL}_3(q)$ module and let $V_{\#}=(\mathbb{F}_q)^6$ denote the minimal module of $G_2(q)$ over $\mathbb{F}_q$. Take $h\in {\rm SL}_3(q)\leqs H$ to be an element of order $2$ with Jordan form $[J_2,J_1]$ on $W$. Then by Lemma \ref{l:g2onmin}, $h$ has Jordan form $[J_2^2,J_1^2]$ on $V_\#$ and thus $h$ has Jordan form $[J_2^{d/3},J_1^{d/3}]$ on $V$. Therefore Corollary \ref{c:ae} implies that $G$ is not almost elusive. 
\end{proof}

This concludes the proof of Theorem \ref{t:mainaffine}.

\section{Quasiprimitive Groups}\label{s:quasi}

In this section we complete the proof of the classification of all quasiprimitive almost elusive groups by proving Theorem \ref{t:mainquasi}. By \cite[Theorem 1]{BHall} a quasiprimitive group is almost elusive only if it is an almost simple or a 2-transitive affine group. Every 2-transitive affine group is primitive, so in view of Theorem \ref{t:fullaeprim} we may assume $G$ is almost simple and imprimitive. 

For the remainder of this section we assume that $G$ is a finite almost simple group with socle $G_0$ and $H$ is a core-free non-maximal subgroup of $G$ such that $G=G_0H$. In particular, we may embed $H$ in a maximal subgroup $M$ of $G$. It is easy to show that $M$ must in fact be core-free.

\begin{lem}\label{l:core-free}
Every maximal overgroup of $H$ in $G$ is core-free. 
\end{lem}
\begin{proof}
Let $M$ be a maximal subgroup of $G$ such that $H<M$. Suppose for contradiction that $M$ is not core-free. Then $G_0$ is a subgroup of $M$ since $G_0$ is the unique minimal normal subgroup of $G$. It follows that $G_0H\leqs M$. However $G$ is quasiprimitive, so $G_0$ is a transitive subgroup of $G$ and thus $G=G_0H$. This a contradiction since $M<G$.  
\end{proof}

In particular, this means that $G$ acts primitively on the set of right cosets of $M$. Recall that $x\in G$ is a derangement if and only if $x^G\cap H=\emptyset$, where $x^G$ denotes the conjugacy class of $x$ in $G$.

\begin{lem}\label{l:reducquas}
Suppose that $(G,H)$ is almost elusive and $H<M$ with $M$ maximal in $G$. Then $(G,M)$ is almost elusive and is therefore one of the cases in Table \ref{tab:fullastab1} or \ref{tab:fullastab2}.
\end{lem}
\begin{proof}
Suppose $(G,M)$ is neither almost elusive nor elusive. Then there exist distinct conjugacy classes $x^G$ and $y^G$ of elements of prime order in $G$ such that $x^G\cap M=\emptyset $ and $y^G\cap M=\emptyset$. Since $H<M$, there are at least two distinct conjugacy classes of derangements of prime order in $G$, which is a contradiction. Thus either $(G,M)$ is almost elusive and thus by Theorem \ref{t:fullaeprim} is found in Table \ref{tab:fullastab1} or \ref{tab:fullastab2}, or $(G,M)$ is elusive and thus by \cite{G} $(G,M)= ({\rm M}_{11},\Li_2(11))$. It is a simple calculation in {\sc Magma} to show that there are there are no almost elusive cases if $(G,M)= ({\rm M}_{11},\Li_2(11))$. The result follows. 
\end{proof}

To complete the classification of the quasiprimitive almost elusive groups it now remains to handle the cases in which $(G,M)$ is contained in Table \ref{tab:fullastab1} or \ref{tab:fullastab2}.

\subsection{Preliminary results}\label{ss:quasiprelim}

In this section we present some preliminary results for the proof of Theorem \ref{t:mainquasi}. 
We remind the reader that $G$ is a finite almost simple group with socle $G_0$ and $H$ is a core-free non-maximal subgroup of $G$ such that $G=G_0H$. Additionally, $H<M$ such that $M$ is a core-free maximal subgroup of $G$.
We begin by stating a number theoretic lemma, \cite[Lemma 2.6]{BTV}.

\begin{lem}\label{l:btv}
Let $r$ and $s$ be primes and let $v$ and $w$ be positive integers. If $r^v + 1 =s^w$ then one of the following holds:
\begin{itemize}\addtolength{\itemsep}{0.2\baselineskip}
    \item[{\rm (i)}] $(r,s,v,w)=(2,3,3,2)$.
    \item[{\rm (ii)}] $(r,w)=(2,1)$ and $s=2^v+1$ is a Fermat prime. 
    \item[{\rm (iii)}] $(s,v)=(2,1)$ and $r=2^w-1$ is a Mersenne prime. 
\end{itemize}
\end{lem}

We recall that for any subgroup $K<G$ we define $K_0=K\cap G_0$. Additionally, $\alpha(K)$ denotes the set of distinct prime divisors of $|K|$ and $\pi(K)=|\alpha(K)|$.

\begin{lem}\label{l:piMpiH}
Assume $(G,M)$ is almost elusive with $\pi(G_0)=\pi(M_0)+1$. Then $(G,H)$ is almost elusive only if $\pi(M_0)=\pi(H_0)$.
\end{lem}
\begin{proof}
Assume that $(G,H)$ is almost elusive. We note that if $r$ is a prime dividing $|G_0|$ but not $|H_0|$ then every element in $G_0$ of order $r$ is a derangement. Thus $\pi(G_0)-\pi(H_0)\leqs 1$. Therefore the result follows easily since $\pi(M_0)\geqs \pi(H_0)$. 
\end{proof}

Before we state our next result, we first provide a brief description of the conjugacy classes of prime order semisimple elements in the group $G={\rm PGL}_2(q)$, where $q=p^f$ is a prime power. 
We first look at the semisimple involutions (in this case we must take $q$ to be odd). In $G$ there are 2 distinct classes of involutions represented by $t_1$ and $t_1'$, using the notation from \cite[Section 3.2.2]{BG_book} which is consistent with \cite[Table 4.5.1]{GLS}. The element $t_1$ lifts to an involution in ${\rm GL}_2(q)$, while $t_1'$ lifts to an irreducible element of order 4. In particular, we note that $G_0$ has a unique class of involutions. Furthermore,  $t_1\in G_0$ if and only if $q\equiv 1\imod 4$ and $t_1'\in G_0$ if and only if $q\equiv 3\imod 4$. See \cite[Section 3.2.2]{BG_book} for more details.

For the remainder of the discussion we look at odd prime order semisimple elements.
We note that $x^{G_0}=x^{G}$, where $x\in G$ is an element of odd prime order $r\neq p$ (see \cite[Theorem 4.2.2(j)]{GLS}) and we will use the notation discussed here in the proof of Lemma \ref{l:conjclassinter}.
Let $r$ be an odd prime divisor of $|G|$ such that $r\neq p$. Then either $r$ divides $q-1$ or $q+1$. 
We define $k=1$ if $r$ divides $q-1$ and $k=2$ otherwise. Take $x\in G$ to be an element of order $r$. Then up to conjugacy $x$ lifts to an element $\hat{x}\in{\rm GL}_2(q)$ that is diagonalisable over $\mathbb{F}_{q^k}$, with eigenvalues $[\lambda^{i},\lambda^{-i}]$ if $k=1$ and $[\lambda^i,\lambda^{qi}]$ if $k=2$, where $\lambda$ is a nontrivial $r^{th}$ root of unity in $\mathbb{F}_{q^k}$ and $1\leqs i \leqs (r-1)/2$. The $G$-classes of elements of order $r$ are uniquely determined by these eigenvalue sets. 
Thus there are $(r-1)/2$ such $G$-classes of elements of order $r$ in $G$ for both $k=1$ and 2. We abuse notation and write the representatives of the $(r-1)/2$ distinct $G$-classes as $[\Lambda]Z$ where $\Lambda=[\lambda^{i},\lambda^{-i}]$ if $k=1$ and $[\lambda^i,\lambda^{qi}]$ otherwise, with $1\leqs i\leqs (r-1)/2$, and $Z$ is the centre of ${\rm GL}_2(q)$. See \cite[Section 3.2.1]{BG_book} for more details.

\begin{lem}\label{l:conjclassinter}
Let $G=\Li_2(q)$ with $q=p^f\geqs 4$ a prime power and let $H$ be a subgroup of $G$. Suppose $r\neq p$ is an odd prime divisor of $|H|$ and take $x\in G$ to be an element of order $r$. Then $x^G\cap H\neq \emptyset$. 
\end{lem}
\begin{proof}
We note that $|G|=\frac{q}{(2,q-1)}(q-1)(q+1)$ and so either $r$ divides $q-1$ or $q+1$. The two cases are very similar, so we only provide details in the case where $r$ divides $q-1$. Let $y\in H$ be an element of order $r$. Then with out loss of generality we can assume $y\in ([\lambda,\lambda^{-1}]Z)^G$. This implies that $y^t\in([\lambda^t,\lambda^{-t}]Z)^G$ for all $1\leqs t\leqs (r-1)/2$, so $\langle y\rangle$ intersects all $G$-classes of elements of order $r$. Thus the result follows since $\langle y\rangle\leqs H$.
\end{proof}

\begin{cor}\label{c:L2pOddDer}
Let $G$ be the almost simple group $\Li_2(q)$ or ${\rm PGL}_2(q)$, where $q=p^f\geqs 4$ is a prime power. Let $H$ be a core-free subgroup of $G$ and suppose $r\neq p$ is an odd prime divisor of $|H|$. Then $(G,H)$ contains no derangements of order $r$.
\end{cor}

\subsection{Proof of Theorem \ref{t:mainquasi}} 

We are now ready to prove Theorem \ref{t:mainquasi}. We recall that $G$ is an almost simple group with socle $G_0$ and a non-maximal core-free subgroup $H$ such that $G=G_0H$. Additionally, $H<M$ where $M$ is a core-free maximal subgroup of $G$. By Lemma \ref{l:reducquas}, we may assume that $(G,M)$ is found in Table \ref{tab:fullastab1} or \ref{tab:fullastab2}.  We begin this section by proving Theorem \ref{t:mainquasi} in some small cases. For this let 
\[
\mathcal{A}=\{A_n,\Li_2(q),{\rm L}^{\epsilon}_3(q'),\Un_4(q'),{\rm PSp}_4(q'),\Un_5(2),\,\Un_6(2),\,{\rm PSp}_6(2),\,{}^2F_4(2)',\,G_2(4)\},
\]
where $n\leqs 20$, $q\leqs 49$ and $q'\leqs 8$.

\begin{prop}\label{p:smalldimsquasi}
Theorem \ref{t:mainquasi} holds for $G_0\in\mathcal{A}$.
\end{prop}
\begin{proof}
Let $M$ denote a core-free maximal subgroup of $G$ such that $H<M$. Then by Lemma \ref{l:reducquas}, $(G,M)$ is recorded in Table \ref{tab:fullastab1} or \ref{tab:fullastab2}. As in Proposition \ref{p:smalldim} we can use {\sc Magma} to obtain the groups $G$ and $M$ such that $(G,M)$ is recorded in Tables \ref{tab:fullastab1} or \ref{tab:fullastab2}. Using the \verb|MaximalSubgroups| command we obtain a list of representatives of the conjugacy classes of maximal subgroups of $M$. For each maximal subgroup $L$ we first check that $|L||G_0|/|L\cap G_0| = |G|$, which ensures that $G$ acting on the cosets of $L$ is quasiprimitive, and we use \verb|Core| to check $L$ is a core-free subgroup of $G$. We then check that $(G,L)$ is almost elusive using the same function described in Proposition \ref{p:smalldim}. If $(G,L)$ is not almost elusive then it is discarded. However if $(G,L)$ is almost elusive, then we look at the maximal subgroups of $L$ and repeat this process until we no longer find almost elusive examples. 
\end{proof}

\begin{rmk}
Take $G_0\in\mathcal{A}$. As discussed in Remark \ref{r:quasirmk}, there may be multiple classes of subgroups with representatives isomorphic to $H$ and it may be the case that not all of these classes lead to an almost elusive example. In Table \ref{tab:QuasiInA} for each pair $(G,H)$ there exists a subgroup $K<G$ such that $K\cong H$ and $(G,K)$ is almost elusive. We record the number of classes of subgroups of $G$ that have representatives $K\cong H$ such that $G=G_0K$ and additionally how many of these classes give almost elusive examples. Additionally, if $(G,H)$ is recorded in Table \ref{tab:imprimitiveAE}, then $(G,K)$ is almost elusive for any subgroup $K<G$ such that $K\cong H$. See Section \ref{s:tables} and Remark \ref{r:main} for more details on these cases. 
\end{rmk}

This handles all the cases in which $(G,M)$ is found in Table \ref{tab:fullastab2}. Thus we may now assume that $(G,M)$ is contained in Table \ref{tab:fullastab1} and $G_0\not\in\mathcal{A}$. We go through each of these cases in turn, using the labels in Table \ref{tab:fullastab1} to denote the cases. Recall we use $\alpha(X)$ to represent the set of prime divisors of $|X|$, so $\pi(X)=|\alpha(X)|$. Additionally we remind the reader that $H_0=H\cap G$. We begin with a remark discussing the case U1 in Table \ref{tab:fullastab1}.

\begin{rmk}\label{r:caseU1}
As discussed in Remark \ref{r:rmkfullastab1}(b) and \cite[Section 5.2.3]{Hall} we do not anticipate that any genuine examples of primitive almost elusive groups arise in case U1 of Table \ref{tab:fullastab1}. Thus we do not anticipate any quasiprimitive almost elusive groups in this case either. However, if such a primitive group did arise, then we are able to construct quasiprimitive almost elusive groups such that the point stabiliser $H$ is a core-free non-maximal subgroup of $G$. For instance, assume that $G=\Un_n(2^f).C_{2f}$, $M$ is the stabiliser of a 1-dimensional non-degenerate subspace of the natural module and $(G,M)$ is as in case U1. Take $H=M_0.C_j$, where $M_0=M\cap G_0$ and $j$ divides $2f$. Then $(G,H)$ is almost elusive if and only if $\pi(j)=\pi(2f)$. 
\end{rmk}

\begin{prop}\label{p:L1L2}
Theorem \ref{t:mainquasi} holds for $(G,M)$ as in cases L1 or L2 in Table \ref{tab:fullastab1}. 
\end{prop}
\begin{proof}
Here $G_0={\rm L}_2(p)$ and $M=C_p{:}C_{k(p-1)/2}$ is a $P_1$ parabolic subgroup of $G$ (that is, $M$ is a Borel subgroup of $G$), where $k=|G:G_0|$ and $p=2^m-1$ is a Mersenne prime in case L1, and $p=2.3^a-1$ is a prime with $a\geqs 2$ in case L2.

We begin by handling case L1, in which case $G=G_0$ or ${\rm PGL}_2(p)$. Here we show that $(G,H)$ is almost elusive only if it is recorded in case II or III of Table \ref{tab:imprimitiveAE}. Set $M_0=M\cap G_0=C_p{:}C_{(p-1)/2}$. We note that since $p\equiv 3\imod 4$, $|M_0|$ is odd and thus $\alpha(G_0)\setminus\alpha(M_0)=\{2\}$. Therefore every involution in $G_0$ is a derangement and as discussed in Section \ref{ss:quasiprelim} there is a unique class of involutions in $G_0$. 
Assume $(G,H)$ is almost elusive. We recall from the discussion in Section \ref{ss:quasiprelim} that ${\rm PGL}_2(p)$ contains two distinct classes of involutions. One class consists of the involutions in $G_0$, each of which is a derangement, and the other comprises of the involutions in ${\rm PGL}_2(p)\setminus G_0$. Thus we conclude that $|H|$ must be even when $G={\rm PGL}_2(p)$. Now we note that $\pi(M_0)=\pi(H_0)$ by Lemma \ref{l:piMpiH}, which is equivalent to the condition $\alpha(M_0)=\alpha(H_0)$ since $H_0\leqs M_0$. Therefore since $p \equiv 3 \imod 4$ and $|H|$ is even when $k=2$ we must have $H=C_p{:}C_{d}$, where $d$ is a proper divisor of $k(p-1)/2$ and $\alpha(d)=\alpha(k(p-1)/2)$. 

Finally we turn to case L2, where $G=G_0$. Here we need to show that $(G,H)$ is almost elusive only if it is recorded in case IV of Table \ref{tab:imprimitiveAE}.   
Assume $(G,H)$ is almost elusive. Then as above we have $\alpha(M_0)=\alpha(H_0)$, so $H=C_p{:}C_{d}$ where $d$ is a proper divisor of $(p-1)/2$ and $\alpha(d)=\alpha((p-1)/2)$. The result follows. 
\end{proof}

\begin{rmk}\label{r:L1L2rmk}
As stated in Remark \ref{r:quasirmk}(b), if $(G,H)$ is recorded in Table \ref{tab:imprimitiveAE}, then $(G,K)$ is almost elusive for any subgroup $K<G$ such that $K\cong H$. Here we justify this claim in the cases labeled II, III and IV in Table \ref{tab:imprimitiveAE}.
Take $(G,H)$ to be as in case II or III of Table \ref{tab:imprimitiveAE}. 
That is $G=G_0=\Li_2(p)$ or ${\rm PGL}_2(p)$ with $p=2^m-1$ a Mersenne prime, and $H=C_p{:}C_d$ such that $d$ is a proper divisor of $k(p-1)/2$ and $\alpha(d)=\alpha(k(p-1)/2)$, where $k=|G:G_0|$. 
Then $|G:H|=2^{m-1}k(p-1)/d$. By Corollary \ref{c:L2pOddDer} there are no derangements of order $r$ for any odd prime divisor $r$ of $p-1$. Since $|H|$ is even when $k=2$ and $|H_0|$ is odd, there exists a unique class of involutory derangements in $G$. Thus $(G,H)$ is almost elusive. 
A very similar argument applies in case IV. Here $G=\Li_2(p)$ and $H=C_p{:}C_d$, where $p=2.3^a-1\geqs 17$ is a prime and $d$ is a proper divisor of $(p-1)/2$ with $\alpha(d)=\alpha((p-1)/2)$. 
Then $|G:H|=3^a(p-1)/d$. By Corollary \ref{c:L2pOddDer} there are no derangements of order $r$ for any odd prime divisor $r$ of $p-1$. Since $3\not\in \alpha(H)$ and there exists a unique class of elements of order 3 in $G$ (see \cite[Proposition 3.2.1]{BG_book}), we conclude that $(G,H)$ is almost elusive.  
\end{rmk}

The case L3 in Table \ref{tab:fullastab1} comes with a variety of number theoretic conditions. We are in fact able to show that all these number theoretic conditions hold only in 2 cases. 
\begin{lem}\label{l:L3}
Assume $(G,M)$ is as in case L3 of Table \ref{tab:fullastab1}. Then $q=9$ or $49$.
\end{lem}
\begin{proof}
Here $G_0=\Li_2(q)$, $G=G_0.f$ , $M$ is a $P_1$ parabolic subgroup of $G$ and the following conditions hold:
\begin{itemize}
\item[{\rm (a)}] $q=2r^z-1$ with $z\geqs 1$,
\item[{\rm (b)}] $r=2^m+1$ is a Fermat prime with $m\geqs 2$ a 2-power; and 
\item[{\rm (c)}] $q=p^f$ with $p$ a prime and $f=2^{m-1}$.
\end{itemize}

We note first that (a) implies that $p^f+1=2r^z$. Let $x=p^{f/2}$ (noting that $f\geqs 2$ is a power of 2). Then the equation becomes 
\begin{equation}\label{e:zlt}
x^2+1=2r^z.
\end{equation}
By \cite[Lemma 2.6]{ZLT} for $z>2$ there are no solutions to \eqref{e:zlt} such that $r$ is a Fermat prime. Thus we may assume that $z\in\{1,2\}$. From here it is a simple calculation to show that the only solutions are $(x,r,z)=(3,5,1)$ and $(7,5,2)$. That is the only cases that arise for L3 are $q=9$ and $q=49$.
\end{proof}

\begin{prop}
Theorem \ref{t:mainquasi} holds for $(G,M)$ as in case L3 in Table \ref{tab:fullastab1}.
\end{prop}
\begin{proof}
By Lemma \ref{l:L3} this case has already been handled in Proposition \ref{p:smalldimsquasi}.
\end{proof}

\begin{prop}
Theorem \ref{t:mainquasi} holds for $(G,M)$ as in case L4 or L5 in Table \ref{tab:fullastab1}.
\end{prop}
\begin{proof}
Here $G={\rm PGL}_2(p)$, where $p=2^m+\epsilon$ is a prime and $M=D_{2(p+\epsilon)}$ with $\epsilon=\pm 1$. We note that $p+\epsilon \equiv 2 \imod 4$ and we set $M_0=M\cap G_0=D_{p+\epsilon}$. Additionally we note that $\alpha(G)\setminus\alpha(M)=\{p\}$. Therefore every element of order $p$ in $G$ is a derangement and there is a unique $G$-class of such elements since $G={\rm PGL}_2(p)$ (see \cite[Proposition 3.2.6]{BG_book}).  
We need to show that $(G,H)$ is almost elusive only if it is recorded in case I of Table \ref{tab:imprimitiveAE}.

Assume that $(G,H)$ is almost elusive. By Lemma \ref{l:piMpiH} we must have $\pi(H_0)=\pi(M_0)$, which is equivalent to $\alpha(H_0)=\alpha(M_0)$ since $H_0\leqs M_0$. We remind the reader that every subgroup of a dihedral group is either cyclic or dihedral. Therefore, $H=D_{2d}$ and $d$ is a proper divisor of $p+\epsilon$ with $\alpha(d)=\alpha(k(p+\epsilon)/2)$, where we set $k=2$ if $d$ is even and 1 otherwise (note that $H_0=D_{2d/k}$). If $d$ is odd, then $H$ has a unique conjugacy class of involutions, but we recall that there are two such classes in $G$, so $d$ must be even. That is $\alpha(d)=\alpha(p+\epsilon)$ and the result follows. 
\end{proof}

\begin{rmk}\label{r:L4L5rmk}
Here we justify Remark \ref{r:quasirmk}(b) when $(G,H)$ is given as in case I of Table \ref{tab:imprimitiveAE}. Here $G={\rm PGL}_2(p)$ and $H=D_{2d}$, where $p=2^m+\epsilon$ is a prime and $d$ is a proper divisor of $p+\epsilon$ with $\alpha(d)=\alpha(p+\epsilon)$. Note that $|G:H|=2^{m-1}p(p+\epsilon)/d$. By Corollary \ref{c:L2pOddDer} there are no derangements of order $r$ for any odd prime divisor $r$ of $p+\epsilon$. Since $d$ is even and $H_0=D_d$, we see that $H\setminus H_0$ contains involutions and so every involution in $G$ has a fixed point. Finally, since $p\not\in\alpha(H)$ and $G$ contains a unique conjugcay class of elements of order $p$ (see \cite[Proposition 3.2.6]{BG_book}), we conclude that $(G,H)$ is almost elusive. 
\end{rmk}

\begin{prop}
Theorem \ref{t:mainquasi} holds for $(G,M)$ as in case A1, A4 or A5 of Table \ref{tab:fullastab1}.
\end{prop}
\begin{proof}
Here $(G,M)=(A_n,A_{n-1})$ or $(S_n,S_{n-1})$, where $n=r^a$ with $r$ a prime or $n=2r^a$ with $r\geqs 3$ a prime (see Table \ref{tab:fullastab1}). Since $(G,M)$ is almost elusive with a unique class of derangements of order $r$ and $H<M$ we conclude that $(G,H)$ contains derangements of order $r$.  Note that by Proposition \ref{p:smalldimsquasi} we may assume that $n>20$.
To begin the analysis, we will assume that $H$ is a maximal subgroup of $M$ (if $(G,H)$ is not almost elusive when $H$ is maximal in $M$ then we are done), where we view $M$ as the stabiliser in $G$ of $n\in\{1,\dots,n\}$. 
Therefore $H$ either acts intransitively, imprimitively or primitively on $\{1,\dots,n-1\}$.
 By Lemma \ref{l:piMpiH} we may assume that $\pi(M_0)=\pi(H_0)$ and so by \cite[Corollary 5]{LPS}, $H$ must act intransitively on $\{1,\dots,n-1\}$. That is $H=(S_k\times S_{n-1-k})\cap M$ for some $1\leqs k< (n-1)/2$. 
In particular, $H$ is the stabiliser in $G$ of the partition $$\{1,\dots, k\}\cup\{k+1,\dots,n-1\}\cup \{n\}.$$ We will show that in each case $(G,H)$ also contains derangements of prime order $s$ such that $s\neq r$ and so $(G,H)$ is not almost elusive. 

Assume first that $k\geqs 4$. The result follows from inspection of the proof of \cite[Lemma 3.8]{BHall}. However we provide the details for completeness. 
First observe that $|G:H|=n\binom{n-1}{k}$. By \cite[Proposition 3.6]{BHall}, $\binom{n-1}{k}$ is divisible by a prime $s$ such that $s>k$ and $s$ does not divide $n$. Thus $s\neq r$ and $s$ divides $n-1-t$ for some $t\in\{0,1,\dots,k-1\}$. Consider an element $g\in G$ with cycle shape $[s^{(n-1-t)/s},1^{t+1}]$. Since $t<k$ it follows that $g\not\in H$, so $g$ is a derangement. 

Next let us assume that $k=1$ or 2. Suppose $s$ is an odd prime divisor of $n-1$. Then every element in $G$ with cycle shape $[s^{(n-1)/s},1]$ is a derangement. Thus we may assume that $n-1=2^t$ for some $t\geqs 5$ (recall we are assuming that $n>20$). This implies $n=r$ is a Fermat prime by Lemma \ref{l:btv}, in which case $G=S_n$ and $M=S_{n-1}$ (see Table \ref{tab:fullastab1}). Suppose first $k=1$. Then any element in $G$ with cycle shape $[2^{(n-1)/2},1]$ is a derangement. Finally assume $k=2$ and take $s$ to be an odd prime divisor of $n-2$. Then $s\neq r$ and any element in $G$ with cycle shape $[s^{(n-2)/s},1^2]$ is a derangement.

Finally we assume $k=3$. Suppose there exist prime divisors $s_1$ and $s_2$ of $n-1$ and $n-2$ respectively, such that $s_1,s_2\geqs 5$. Then any element in $G$ with cycle shape $[s_1^{(n-1)/s_1},1]$ or $[s_2^{(n-2)/s_2},1^2]$ is a derangement. Thus $G$ is almost elusive only if $n-1$ and $n-2$ are both indivisible by primes larger than 3. Let us assume that both $n-1$ and $n-2$ are only divisible by the primes 2 and 3. Then either $(n-1,n-2)=(2^m,3^b)$ or $(3^b,2^m)$ for some $b\geqs 3$ and $m\geqs 5$ (recall $n>20$). 
Thus by Lemma \ref{l:btv} this never occurs. The result follows.
\end{proof}

\begin{prop}
Theorem \ref{t:mainquasi} holds for $(G,M)$ as in case A2 or A3 of Table \ref{tab:fullastab1}.
\end{prop}
\begin{proof}
In cases A2 and A3 of Table \ref{tab:fullastab1} $G=S_n$, $M=S_{n-2}\times S_2$ and either $n=2^m$ and $n-1=r$ is a Mersenne prime (case A2) or $n=2^m+1=r$ is a Fermat prime (case A3). 
Since $(G,M)$ is almost elusive with a unique class of derangements of order $r$ and $H<M$ we conclude that $(G,H)$ contains derangements of order $r$. 
We begin by assuming that $H$ is maximal in $M$ (if $(G,H)$ is not almost elusive when $H$ is maximal in $M$ then we are done). The maximal subgroups of $M$ are as follows
\begin{itemize}
\item[{\rm (i)}] $(A_{n-2}\times 1).2$
\item[{\rm (ii)}] $S_{n-2}$
\item[{\rm (iii)}] $L\times S_2$ where $L$ is a maximal subgroup of $S_{n-2}$.
\end{itemize}
See \cite[Lemma 1.3]{MSDP}.
First suppose that $H$ is as in case (i). Then up to conjugacy $H=\langle A_{n-2}, (n-3, n-2)(n-1, n) \rangle\leqs G_0$. This is a contradiction since $G=G_0H$.

Now suppose $H$ is as in case (ii).
Here $H=S_{n-2}$ and every element of $H$ has cycle shape $[a,1^2]$, where $a$ is a partition of $n-2$. If $n=2^m$, then any involution in $G$ with cycle shape $[2^{n/2}]$ is a derangement. Similarly, if $n=2^m+1$ is a Fermat prime, then any involution in $G$ with cycle shape $[2^{(n-1)/2},1]$ is a derangement. Thus $G$ contains derangements of order $r$ and of order $2$, so $(G,H)$ is not almost elusive.

Finally suppose $H$ is as in case (iii).
By Lemma \ref{l:piMpiH} we may assume $\pi(M_0)=\pi(H_0)$, so $|H_0|$ is divisible by every prime $p\leqs n-2$. Thus $|L|$ must be divisible by the two largest primes not exceeding $n-2$. Therefore by \cite[Theorem 4]{LPS} either $L=A_{n-2}$ or $L=S_k\times S_{n-2-k}$ for some $1\leqs k< (n-2)/2$. We recall that $G$ contains a conjugacy class of derangements of order $r$  (where $r=n-1$ when $n=2^m$, and $r=n$ when $n$ is a Fermat prime). We show that in each of these cases $(G,H)$ also contains derangements of a distinct prime order and so $(G,H)$ is not almost elusive. 

Suppose first that $H=A_{n-2}\times S_2$. If $n=2^m$, then any involution in $G$ with cycle shape $[2^{n/2}]$ is a derangement (since $n/2$ is even). Similarly, if $n=2^m+1$ is a Fermat prime, then any involution in $G$ with cycle shape $[2^{(n-1)/2},1]$ is a derangement (since $(n-1)/2$ is even).

For the remainder of the proof we may assume $H=S_k\times S_{n-2-k}\times S_2$ for some $1\leqs k\leqs (n-2)/2$. Up to conjugacy $H$ is the stabiliser of the partition $$\{1,\dots,k\}\cup\{k+1,\dots,n-2\}\cup\{n-1,n\}.$$ 
Suppose $k\geqs 4$. By \cite[Proposition 3.6]{BHall} there exists a prime $s>k$ that divides $\binom{n-2}{k}$, which means that $s\neq r$ and $s$ divides $n-2-t$ for some $t\in\{0,1,\dots,k-1\}$. Thus any element in $G$ with cycle shape $[s^{(n-2-t)/s}, 1^{t+2}]$ is a derangement. 
Now suppose $k=1$ or 2. Let $s$ be an odd prime divisor of $n-2$. Then any element in $G$ with cycle shape $[s^{(n-2)/s},1^2]$ is a derangement. 
Finally suppose $k=3$. Assume first that $n=2^m$ and $n-1$ is a Mersenne prime. Then any element with cycle shape $[2^{n/2}]$ is a derangement. Now assume $n=r=2^m+1$ is a Fermat prime. Then $n-2=2^{m}-1$ is odd and by Lemma \ref{l:btv} there exists a prime $s\geqs 5$ such that $s$ divides $n-2$ (recall $n>20$). Thus any element in $G$ with cycle shape $[s^{(n-2)/s},1^2]$ is a derangement. The result follows. 
\end{proof}

This completes the proof of Theorem \ref{t:mainquasi} and the classification of the quasiprimitive almost elusive groups.

\section{Proof of Corollaries \ref{c:largestprim} and \ref{c:depth}}\label{s:corollaries}
In this section we provide proofs of Corollaries \ref{c:largestprim} and \ref{c:depth}.  
 
\subsection{Proof of Corollary \ref{c:largestprim}}
Let $G\leqs{\rm Sym}(\Omega)$ be a quasiprimitive almost elusive permutation group with socle $G_0$ and point stabiliser $H$. The result for the affine groups is trivial. Thus in view of Theorems \ref{t:fullaeprim} and \ref{t:mainquasi} we may assume that $G$ is an almost simple group such that $(G,H)$ is one of the cases recorded in Tables \ref{tab:fullastab1}, \ref{tab:fullastab2}, \ref{tab:imprimitiveAE} or \ref{tab:QuasiInA}. The cases in Table \ref{tab:fullastab2} and \ref{tab:QuasiInA} can be handled easily using computational methods in {\sc Magma} to calculate the degree of $G$. Thus it remains to handle the cases in Tables \ref{tab:fullastab1} and \ref{tab:imprimitiveAE}. Let $s$ be the order of the elements in the unique $G$-class of derangements of prime order. 

Suppose $(G,H)$ is recorded in Table \ref{tab:fullastab1}. The proof is similar in all cases so we will only show the details for cases U1, L4 and A2 (we note we use \cite[Chapters 3 and 4]{KL} for the orders of the classical groups). 

First assume $(G,H)$ is as in case U1 in Table \ref{tab:fullastab1}. Here $G=\Un_n(q).[2f]$ such that $q=2^f$ is even, $n\geqs 5$ is a prime divisor of $q+1$ and $s=2nf+1$ is the unique primitive prime divisor of $q^{2n}-1$. Additionally, $H$ is the stabiliser of a 1-dimensional non-degenerate subspace of the natural module. By \cite[Proposition 4.1.4]{KL} we have
$$|G_0|=\frac{1}{n}q^{n(n-1)/2}\prod_{i=2}^n(q^i-(-1)^i) \mbox{ and } |H_0|=\frac{1}{n}q^{(n-1)(n-2)/2}\prod_{i=1}^{n-1}(q^i-(-1)^i).$$
Since $s$ is the unique primitive prime divisor of $q^{2n}-1$, by \cite[Remark 2.10]{Hall} we deduce that
$$|\Omega|=q^{n-1}\frac{q^{n}+1}{q+1}=q^{n-1}.n.s^l$$
for some $l\geqs 1$. Thus the result follows since $s=2nf+1> n$. 

Next assume that $(G,H)$ is as in case L4 of Table \ref{tab:fullastab1}. Then $G={\rm PGL}_2(p)$ and $H=D_{2(p-1)}$ where $s=p=2^m-1$ is a Mersenne prime. Here $|G|=2^{m}p(p-1)$ and $|H|=2(p-1)$. Thus $|\Omega|= 2^{m-1}p$. The result follows since $p\geqs 7$. Finally let us assume that $(G,H)$ is as in case A2 of Table \ref{tab:fullastab1}. Then $G=S_n$ and $H=S_{n-2}\times S_2$ with $n=2^m=s+1$. Thus $|\Omega|= (n(n-1))/2=2^{m-1}s$. The result follows since $s>2$.

Finally suppose that $(G,H)$ is recorded in Table \ref{tab:imprimitiveAE}. Assume first that $(G,H)$ is as in case II or III. Then $G=\Li_2(p)$ or ${\rm PGL}_2(p)$ with $p=2^m-1$ a Mersenne prime and $H=C_p{:}C_d$, where $d$ is a proper divisor of $k(p-1)/2$ and $\alpha(d)=\alpha(k(p-1)/2)$ with $k=|G:G_0|$. Here $s=2$ and $|\Omega|=2^{m-1}k(p-1)/d$. Thus $2$ is not the largest prime divisor of $|\Omega|$ since $(p-1)/d$ is divisible by an odd prime. Next assume that $(G,H)$ is as in case I. Then $G={\rm PGL}_2(p)$ where $s=p=2^m+\epsilon$ is a prime and $\epsilon=\pm 1$. Additionally, $H=D_{2d}$ where $d$ is a proper divisor of $p+\epsilon$ and $\alpha(d)=\alpha(p+\epsilon)$. Thus $|\Omega|= 2^{m-1}p(p+\epsilon)/d$. Since all prime divisors of $p+\epsilon$ must be less than $p$ we conclude that $p$ is the largest prime divisor of $|\Omega|$. Finally assume that $(G,H)$ is as in case IV. Here $G=\Li_2(p)$ where $p=2.3^a-1$ is a prime such that $a\geqs 2$ and $H=C_p{:}C_d$ such that $d$ is a proper divisor of $(p-1)/2$ and $\alpha(d)=\alpha((p-1)/2)$. Here $s=3$ and $|\Omega|=3^a.(p-1)/d$. Thus it is clear to see that 3 is not the largest prime divisor of $|\Omega|$ if and only if $(p-1)/d$ has an odd prime divisor. 
\qed

\subsection{Proof of Corollary \ref{c:depth}}
Let $G$ be an almost simple group with socle $G_0$ and suppose $D_G\geqs 2$. Let $H$ be a non-maximal core-free subgroup such that $G=G_0H$ and $(G,H)$ is almost elusive. In view of Theorem \ref{t:mainquasi}, we may assume that $G$ is one of the groups recorded in Table \ref{tab:imprimitiveAE} or \ref{tab:QuasiInA}.

First consider the cases in Table \ref{tab:QuasiInA}. In each case we can use {\sc Magma} to compute $D_G$ precisely, applying similar techniques used in the proof of Proposition \ref{p:smalldimsquasi}.
Finally let us assume $G$ is one of the groups in Table \ref{tab:imprimitiveAE}. The analysis of these cases is very similar, so we only provide details when $G={\rm PGL}_2(p)$ and $p=2^m-1$ is a Mersenne prime. In this case $(G,H)$ is almost elusive only if $H=D_{2d}$ or $C_p{:}C_d$, where $d$ is a proper divisor of $p-1$ and $\alpha(d)=\alpha(p-1)$. In both cases, it is easy to see that $$d_G(H)=\omega(p-1)-\omega(d) + 1.$$
 For example, take $H=C_p{:}C_d<M=C_p{:}C_{p-1}$, where $M$ is a Borel subgroup of $G$. We can compute the depth of $H$ by constructing a chain of subgroups starting with $M$ and then repeatedly taking prime index subgroups, where each subgroup in the chain is of the form $C_p{:}C_t$, where $t$ is a proper divisor of $p-1$ and $\alpha(t)=\alpha(p-1)$.
Thus 
$$D_G = \max d_G(H) = d_G(D_{2e})=d_G(C_p{:}C_{e})$$ 
where $e$ is the product of the distinct prime divisors of $p-1$. That is $\omega(e)=\pi(p-1)$, and we conclude that $D_G= \omega(p-1)-\pi(p-1)+1$ as in part (ii) of the corollary.
\qed

\section{Tables}\label{s:tables}
In this final section we present the tables of the quasiprimitive almost elusive groups, which we have referenced throughout this paper. 
We first provide a little information regarding how to read Tables \ref{tab:fullastab1} and \ref{tab:fullastab2}. 

\begin{rmk}
Here $G\leqs{\rm Sym}(\Omega)$ is an almost simple permutation group with socle $G_0$ and point stabiliser $H$.

\begin{itemize}
\item[{\rm (a)}] In both tables we record the \emph{type of} $H$. For the alternating groups this describes the exact structure of $H$. In the case where $G_0$ is a classical group we additionally record the Aschbacher collection containing $H$. For the geometric collections (those labeled $\mathcal{C}_1,\dots , \mathcal{C}_8$) the type gives the approximate structure of $H\cap {\rm PGL}(V)$, where $V$ is the natural module of $G_0$, and for the non-geometric collection (denoted as $\mathcal{S}$) the type of $H$ denotes the socle of $H$. For example, take $G=\Li_n(q)$ and $H$ to be of type ${\rm GL}_1(q)\wr S_n$. Then $H$ is the stabiliser of a direct sum decomposition of $V=V_1\oplus\dots\oplus V_n$ where ${\rm dim}V_i=1$ for all $1\leqs i \leqs n$. In addition, we adopt the standard notation and use $P_i$ to denote a maximal parabolic subgroup, which is the stabiliser in $G$ of an $i$-dimensional totally singular subspace of $V$. In all other cases the type of $H$ describes the structure of $H\cap G_0$.  

\item[{\rm (b)}] The column labeled $G$ records the groups $G$ for which $(G,H)$ is almost elusive. Additionally, the conditions stated in the tables are in addition to any conditions required for simplicity of $G_0$ and maximality of $H$ (see \cite[Section 3.5]{KL}).

\item[{\rm (c)}] {In the final column, we describe the unique conjugacy class of derangements of prime order in $G$. If there is a unique $G$-class of elements of the given prime order in $G$, then we represent this class with the relevant prime. However, if there are multiple classes of the given prime order then we describe the precise class. 
For example, if $G_0=\Un_3(3)$ with $H$ of type $\Li_2(7)$, then $G_0$ contains two $G$-classes of elements of order 3. As shown in the table, the Jordan form on $V$ of the derangements is $[J_2,J_1]$, where $J_i$ denotes a standard unipotent block of size $i$.

Similarly if $G=\Un_4(2).2$ and $H$ is of type ${\rm Sp}_4(2)$, then $G$ has three $G$-classes of elements of order 3, labeled 3A, 3C and 3D with $|3{\rm A}|=80$, $|3{\rm C}|=240$ and $|3{\rm D}|=480$; the derangements are in 3A. Something similar occurs in the cases where $G_0={\rm PSp}_6(2)$ with $H$ of type ${\rm O}^+_6(2)$ and $G_0={}^2F_4(2)'$ with $H$ of type $\Li_2(25)$. 
In case L1 of Table \ref{tab:fullastab1} the class 2A of involutory derangements in $G$ is precisely the unique class of involutions in $G_0$. 

Finally for the alternating and symmetric groups we represent the conjugacy class by the cycle structure of the element. For example, in the case $(G,H)=(A_6,\Li_2(5))$ the elements in the unique class of derangements of prime order have cycle shape $[3,1^3]$.}
\end{itemize}

\end{rmk}

\begin{table}
\renewcommand\thetable{A1}
\[
\begin{array}{clclllc} \hline
\mbox{Case} &G_0 & &\mbox{Type of } H &G & \mbox{Conditions}&  x \\ \hline\rule{0pt}{2.5ex} 
{\rm U1} &\Un_n(q) & \mathcal{C}_1 & {\rm GU}_1(q)\perp {\rm GU}_{n-1}(q) & G_0.J & \text{see Remark \ref{r:rmkfullastab1}(b)} & 2nf+1\\

{\rm L1} & {\rm L}_{2}(q) &\mathcal{C}_1&P_1 & {\rm PGL}_{2}(q), \, G_0  & q = p = 2^m-1 & \texttt{2A} \\
{\rm L2}&                     &                   & & G_0 & q=p, \, p+1 = 2.3^a, \, a \geqs 2 & 3 \\
{\rm L3} &                     &                   && G_0.f &  \mbox{See Remark \ref{r:rmkfullastab1}(c)} & r \\
{\rm L4}&                     & \mathcal{C}_2 &{\rm GL}_{1}(q) \wr S_2 & {\rm PGL}_{2}(q) &  q=p=2^m-1 & p \\ 
{\rm L5} &                     &\mathcal{C}_3&{\rm GL}_{1}(q^2) & {\rm PGL}_{2}(q) & q=p=2^m+1, \, m\geqs 3 & p \\ 
    
{\rm A1} &A_n & & S_{n-1} & S_{n} & n = r^a,\, a\geqs 1& [r^{n/r}] \\
{\rm A2}&      && S_{n-2} \times S_2 & S_n & n = 2^m = r + 1 & [r,1] \\
{\rm A3} &      && & S_n & n = 2^m+1 = r & [r] \\
{\rm A4} &      && A_{n-1} & A_{n} & \mbox{$n = r^a$, $a \geqs 2$} & [r^{n/r}] \\
{\rm A5}&      && & A_n & n=2r^a,\, r \geqs 3 ,\, a\geqs 2 & [r^{n/r}] \\

\hline
\end{array}
\]
\caption{Primitive almost elusive almost simple groups: Part I}
\label{tab:fullastab1}
\end{table}

\begin{table}
\renewcommand\thetable{A2}
\[
\begin{array}{lcllc} \hline
G_0 & &\mbox{Type of } H &G &   x \\ \hline\rule{0pt}{2.5ex} 
\Li_3(4) &\mathcal{C}_1 &   {\rm GL}_1(4)\oplus{\rm GL}_2(4)&G_0.2_1,\, G_0.2_3,\,G_0.2^2&  7\\
            &\mathcal{C}_5 &   {\rm GL}_3(2) &G_0.2_1,\, G_0.2_2,\,G_0.2^2&   5\\
            &\mathcal{S} &  A_6 &   G_0.2_3 &  7\\
\Li_2(8) &\mathcal{C}_1&P_1 & G_0.3, \, G_0 & 3 \\
                     & \mathcal{C}_2 &{\rm GL}_{1}(q) \wr S_2 & G_0.3, \, G_0 &   3 \\
                     &\mathcal{C}_3&{\rm GL}_{1}(q^2) & G_0.3 &   7 \\ 

\Un_6(2) &\mathcal{C}_5 &  {\rm Sp}_6(2) &G_0.2  & 11\\

             &\mathcal{S} & \Un_4(3) & G_0.2  & 11\\

\Un_5(2) &\mathcal{C}_1 &  {\rm GU}_1(2)\perp {\rm GU}_4(2) &G_0.2 &  11 \\
             &\mathcal{C}_2 & {\rm GU}_1(2)\wr S_5 & G_0.2& 11 \\

\Un_4(3) & \mathcal{C}_1 & P_2 & G_0.2_2 &  7 \\

\Un_4(2) &\mathcal{C}_1 &   P_1 & G_0.2,\, G_0&  5\\
             &\mathcal{C}_2 & {\rm GU}_1(2)\wr S_4 &  G_0.2,\, G_0& 5\\
             &\mathcal{C}_5 & {\rm Sp}_4(2) & G_0.2 & \texttt{3A} \\
{\rm U}_{3}(3) & \mathcal{C}_1& P_1 & G_0.2 &  7 \\      
                      & \mathcal{S} & {\rm L}_{2}(7) & G_0.2, \, G_0 &   [J_2,J_1] \\ 
{\rm U}_{3}(4) & \mathcal{C}_1& {\rm GU}_{2}(q) \times {\rm GU}_{1}(q) & G_0.4 &   13 \\
                      & \mathcal{C}_2 & {\rm GU}_{1}(q) \wr S_3 & G_0.4 &  13 \\                             
{\rm U}_{3}(8) & \mathcal{C}_1& {\rm GU}_{2}(q) \times {\rm GU}_{1}(q) & G_0.6 &   19 \\
{\rm PSp}_6(2) &\mathcal{C}_1 &  {\rm Sp}_2(2)\perp{\rm Sp}_4(2) & G_0 &7\\
                       &\mathcal{C}_8 & {\rm O}^{+}_6(2) & G_0   &  \texttt{3B} \\
                       &\mathcal{C}_8 &  {\rm O}^{-}_6(2) & G_0 &  7 \\
                
{\rm PSp}_4(7) &\mathcal{C}_1 & P_2 & G_0.2,\,G_0 & 5\\

{}^2F_4(2)' &&{\rm L}_{2}(25) & G_0.2,\, G_0 &  \texttt{2A} \\
& &5^2{:}4A_4 & G_0.2 & 13\\
G_2(4) & & {\rm J}_2 & G_0.2 & 13 \\
A_{10} & &(S_7 \times S_3) \cap G & G_0&  [5^2] \\
A_9 & & (S_7 \times S_2) \cap G & G_0.2, G_0 &  [3^3] \\
      & & (S_6 \times S_3) \cap G& G_0.2, G_0 &  [7,1^2] \\
A_6 & &S_3 \wr S_2& G_0.2=S_6 &   [5,1] \\
      & &{\rm L}_{2}(5)  &G_0 &  [3,1^3] \\
      & & D_{20} & G_0.2={\rm PGL}_{2}(9)&  3 \\
      & & 5{:}4  & G_0.2={\rm M}_{10}&  3 \\
      & &  3^2{:}Q_8 & G_0.2={\rm M}_{10}&  5 \\
A_5 & &D_{10}& G_0 &  [3,1^2] \\ 
               
\hline
\end{array}
\]
\caption{Primitive almost elusive almost simple groups: Part II}
\label{tab:fullastab2}
\end{table}

\begin{rmk}\label{r:rmkfullastab1}
Here we provide some additional remarks on the cases arising in Tables \ref{tab:fullastab1} and \ref{tab:fullastab2}.
\begin{itemize}
\item[{\rm (a)}] In view of the isomorphisms 
$$\Li_2(4)\cong\Li_2(5)\cong A_5,\, \Li_2(9)\cong A_6,\, \Li_2(7)\cong \Li_3(2),$$
$$\Un_4(2)\cong{\rm PSp}_4(3),\,G_2(2)'\cong \Un_3(3) ,\,{}^2G_2(3)'\cong \Li_2(8),$$ the tables are complete. See \cite[Proposition 2.9.2 and Theorem 5.1.1]{KL} for a complete list of the isomorphisms between the small dimensional groups of Lie type and alternating groups. 

\item[{\rm (b)}] For case U1 in Table \ref{tab:fullastab1} there are additional conditions. Firstly, we note that $G_0=\Un_n(q)$ with number theoretic restrictions, namely $q=2^f$ is even, $n\geqs 5$ is a prime divisor of $q+1$ and $2nf+1$ is the unique primitive prime divisor of $q^{2n}-1$. Additionally, $G=G_0.J$ where $J\leqs{\rm Out}(G_0)=\langle\ddot{\delta}\rangle{:}\langle\ddot{\phi}\rangle$, $J\cap\langle\ddot{\delta}\rangle={1}$ and $J$ projects onto $\langle\ddot{\phi}\rangle$. Here we use $\delta$ and $\phi$ to denote a diagonal automorphism of order $n$ and a field automorphism of order $2f$ respectively in $\Aut(G_0)$ (see \cite[Section 1.7]{BHR}), and following \cite{KL} for any element $x\in\Aut(G_0)$ we use $\ddot{x}$ to denote the coset $G_0x\in {\rm Out}(G_0)=\Aut(G_0)/ G_0$. As explained in \cite[Remark 1]{Hall} we do not anticipate any genuine examples arise in this case. For more information on the details of this case, see \cite[Section 5.2.3]{Hall}. 

\item[{\rm (c)}] Consider case L3 in Table \ref{tab:fullastab1}. There are two groups of the form $G_0.f$, namely $G_0.\langle \delta\phi\rangle$ and $G_0.\langle \phi \rangle$, where $\phi$ is a field automorphism of order $f$ and $\delta$ is a diagonal automorphism of order 2. We additionally require that $q=p^f=2r^a-1$, where $r=2^m+1$ is a Fermat prime, $m\geqs2$ is a 2-power, $a\geqs 1$ and $f=2^{m-1}$. In Lemma \ref{l:L3} it is shown that these number theoretic conditions imply $q=9$ or 49.

\item[{\rm (d)}]For the cases with $G_0=\Li_3(4)$ in Table \ref{tab:fullastab2} the recorded groups in the column labeled $G$ are defined using Atlas notation \cite{WebAt}. That is $G_0.2_1=\Li_3(4).\langle \iota\phi \rangle $, $G_0.2_2=\Li_3(4).\langle \phi \rangle$, $G_0.2_3=\Li_3(4).\langle \iota\rangle$ and $G_0.2^2=\Li_3(4).\langle \iota,\phi \rangle$, where $\iota$ denotes the inverse-transpose graph automorphism and $\phi$ denotes a field automorphism of order $2$. Similarly for the case with $G_0=\Un_4(3)$ in Table \ref{tab:fullastab2}. Here $G_0.2_2=\Un_4(3).\langle \gamma \rangle $, where $\gamma$ is an involutory graph automorphism and  $C_{G_0}(\gamma)={\rm PSp}_4(3)$. Finally in the case where $G={}^2F_4(2)$ and $H$ is of type $\Li_2(25)$, we note that $H=\Li_2(25).2_3$. We use $\Li_2(25).2_3$ to denote $\Li_2(25).\langle\delta\phi\rangle$ where $\delta$ and $\phi$ are standard involutory diagonal and field automorphisms respectively (see \cite[Section 1.7]{BHR} for example).  
\end{itemize}
\end{rmk}

\begin{table}
\renewcommand\thetable{B1}
\[
\begin{array}{clllcc} \hline
\mbox{Case} & G & H &  \mbox{Conditions}& \alpha(d)&x    \\ \hline\rule{0pt}{2.5ex} 
{\rm I} &{\rm PGL}_2(p) & D_{2d} & p=2^m + \epsilon,\,m>2 &\alpha(p+\epsilon)  & p \\
{\rm II} &			& C_p{:}C_{d} & p=2^m-1& \alpha(p-1) & 2\\
{\rm III} &{\rm L}_2(p) 	& C_p{:}C_{d} & p=2^m-1 & \alpha(\frac{p-1}{2})   & 2\\
{\rm IV} &			& C_p{:}C_{d} & p=2.3^a-1, \, a\geqs 2 &\alpha(\frac{p-1}{2})& 3\\
\hline
\end{array}
\]
\caption{Some quasiprimitive almost elusive groups}
\label{tab:imprimitiveAE}
\end{table}

\begin{table}
\renewcommand\thetable{B2}
\[
\begin{array}{lllcccc} \hline\rule{0pt}{2.5ex} 
G & H & M &  a & b &c & x \\ \hline\rule{0pt}{2.5ex} 
 {\rm M}_{10} & 3^2{:}4 & 3^2{:}Q_8 &  2 & 2 & 2& 5\\
 A_9 &  (A_5\times 3){:}2^{\dagger} & (A_6\times 3){:}2 & 2 & 1 & 2 &7\\
 S_9   & S_5\times {S_3}^{\dagger} &  S_6\times S_3 &  2 & 1 & 2&7\\

 \Li_2(8).3 &  S_3\times 3&D_{18}{:}3 &  1 & 1  & 2 & 7 \\

 \Li_2(49).2_3 & 7^2{:}(3\times Q_8) & {7}^2{:}(3\times Q_{16}) &  2 & 2 & 2 &5  \\
			& 7^2{:}{12} &{7}^2{:}(3\times Q_{16}) &  2 & 2 & 3 &5 \\
 \Un_3(3).2 & 3.(S_3\wr 2) &(3^{1+2}{:}8).2 &  1 & 1 & 2&7 \\
			 & 3.S_3^2&(3^{1+2}{:}8).2 &  1 & 1 & 3&7  \\
 			 &  S_3^2 & (3^{1+2}{:}8).2 &  1 & 1 &4 &7 \\
 \Un_4(2)  	& S_3^2{:}S_3 & 3^3.S_4 &  1 & 1 & 2&5 \\
			& 3 \times S_3^2 & 3^3.S_4 &  1 & 1 & 3&5 \\
			&  6 \times S_3& 3^3.S_4,\,\, 2^{1+4}{:}{\rm SU}_2(2){:}3  & 1 & 1 &4 &5 \\
 			& {\rm SL}_2(3){:}A_4&  2^{1+4}{:}{\rm SU}_2(2){:}3 &  1 & 1 &2 &5 \\

\Un_4(2).2 &  S_3\times(S_3\wr 2) & 3^3{:}S_4\times 2  &  1 & 1 &2 &5 \\
	         &  S_3^3 & 3^3{:}S_4\times 2   & 2 & 1 &3 &5 \\
	         &  2 \times S_3^2 &  3^3{:}S_4\times 2   ,\,\, 2^3{:}A_4.D_{12}   & 3 & 1 &4 &5 \\
		   &  2^{1+4}{:}{\rm SU}_2(2){:}3 & 2^3{:}A_4.D_{12} &  1 & 1 & 2 &5 \\
 \Un_5(2).2 & S_3 \times S_6 &  (3\times {\rm SU}_4(2)){:}2  & 1 & 1  & 2 &11 \\
 {\rm PSp}_6(2) &   S_5\times {\rm Sp}_2(2)^{\dagger} & S_6\times {\rm Sp}_2(2) &  2 & 1 & 2 &7\\
 \hline
\end{array}
\]
\caption{Sporadic quasiprimitive almost elusive groups.}
\label{tab:QuasiInA}
\end{table}

\begin{rmk}\label{r:main}
Here we remark on Tables \ref{tab:imprimitiveAE} and \ref{tab:QuasiInA}.  
\begin{itemize}
\item[{\rm (a)}] For both tables, in the column labeled $G$ we record the almost simple groups of interest. The column labeled $H$ records the core-free non-maximal subgroups of $G$ up to isomorphism for which $(G,H)$ is almost elusive. Additionally, the column labeled $x$ gives the unique conjugacy class of derangements of prime order. Note this is denoted by the relevant prime, $r$, since in all cases there is a unique conjugacy class in $G$ of elements of order $r$. 

\item[{\rm (b)}] In Table \ref{tab:QuasiInA}, the column labeled $a$ records the number of $G$-classes of subgroups $K$ of $G$ such that $G=G_0K$ and $H\cong K$, $b$ records the number of these $G$-classes such that $(G,K)$ is almost elusive and $c=d_G(H)$ denotes the depth of $H$ (see Corollary \ref{c:depth}). Additionally, in the column labeled $M$ we record representatives of the $G$-classes of maximal subgroups of $G$ that contain at least one subgroup $K$ of $G$ with $G=G_0K$, $H\cong K$ such that $(G,K)$ is almost elusive.

\item[{\rm (c)}] For Table \ref{tab:imprimitiveAE}, $p$ denotes a prime and we remind the reader that for a positive integer $t$ we use $\alpha(t)$ to denote the set of distinct prime divisors of $t$. 

\item[{\rm (d)}] In Table \ref{tab:QuasiInA} the $\dagger$ denotes the fact that the relevant $A_5$ or $S_5$ subgroup of $H$ is a primitive subgroup of the corresponding $A_6$ or $S_6$ subgroup of $M$.

\item[{\rm (e)}] In Table \ref{tab:QuasiInA}, we write $G=\Li_2(49).2_3=\Li_2(49).\langle\delta\phi\rangle$, where $\delta$ and $\phi$ denote standard involutory diagonal and field automorphisms of $G_0$, respectively (see \cite[Section 1.7]{BHR}). 
\end{itemize}
\end{rmk}

\end{document}